\documentclass[12pt]{amsart}
\usepackage{epsfig,color}

\usepackage{mathrsfs}

\theoremstyle{definition}

\newtheorem{theorem}{Theorem}[section]
\newtheorem{definition}[theorem]{Definition}
\newtheorem{conjecture}[theorem]{Conjecture}
\newtheorem{proposition}[theorem]{Proposition}
\newtheorem{lemma}[theorem]{Lemma}
\newtheorem{remark}[theorem]{Remark}
\newtheorem{corollary}[theorem]{Corollary}

\newtheorem*{acknowledgements}{Acknowledgements}

\numberwithin{equation}{section}

\DeclareMathOperator*{\supp}{spt}
\renewcommand\div{\operatorname{div}}

\newcommand{\e}{\operatorname{e}}
\newcommand{\wt}{\widetilde}
\newcommand{\pr}{\partial}

\newcommand{\ol}{\overline}
\newcommand{\Lap}{\Delta}
\newcommand{\ssubset}{\subset\subset}

\newcommand{\reg}{\operatorname{reg}}
\newcommand{\sing}{\operatorname{sing}}

\renewcommand*\d{\mathop{}\!\mathrm{d}}

\newcommand{\mr}{\mathbin{\vrule height 1.6ex depth 0pt width
0.13ex\vrule height 0.13ex depth 0pt width 1.3ex}}

\DeclareMathOperator*{\vol}{Vol}

\DeclareMathOperator*{\erf}{erf}

\title{On the entropy of closed hypersurfaces and singular self-shrinkers}

\author{Jonathan J. Zhu}%
\address{Department of Mathematics,
Harvard University, Cambridge, MA 02138, USA}
\email{jjzhu@math.harvard.edu}

\begin{document}
\date{\today}
\maketitle

\begin{abstract}
Self-shrinkers are the special solutions of mean curvature flow in $\mathbf{R}^{n+1}$ that evolve by shrinking homothetically; they serve as singularity models for the flow. The entropy of a hypersurface introduced by Colding-Minicozzi is a Lyapunov functional for the mean curvature flow, and is fundamental to their theory of generic mean curvature flow.

 In this paper we prove that a conjecture of Colding-Ilmanen-Minicozzi-White, namely that any closed hypersurface in $\mathbf{R}^{n+1}$ has entropy at least that of the round sphere, holds in any dimension $n$. This result had previously been established for the cases $n\leq 6$ by Bernstein-Wang using a carefully constructed weak flow.
 
The main technical result of this paper is an extension of Colding-Minicozzi's classification of entropy-stable self-shrinkers to the singular setting. In particular, we show that any entropy-stable self-shrinker whose singular set satisfies Wickramasekera's $\alpha$-structural hypothesis must be a round cylinder $\mathbf{S}^k(\sqrt{2k})\times \mathbf{R}^{n-k}$. 
\end{abstract}

%\tableofcontents

\section{Introduction}
\label{sec:introduction}

Let $\Sigma^n$ denote a hypersurface in $\mathbf{R}^{n+1}$. In \cite{CMgeneric}, Colding and Minicozzi introduced the entropy functional for such hypersurfaces, defined by \begin{equation} \Lambda(\Sigma) = \sup_{\substack{x_0\in\mathbf{R}^{n+1}\\t_0>0}} F(t_0^{-1}(\Sigma - x_0)),\end{equation} 
where the $F$-functional is the Gaussian area
\begin{equation}F(\Sigma)= (4\pi)^{-\frac{n}{2}} \int_\Sigma \e^{-\frac{|x|^2}{4}}.\end{equation}
A family of hypersurfaces is said to flow by mean curvature if \begin{equation} \pr_t x = \vec{H},\end{equation} where $x$ is the position vector and $\vec{H}$ is the mean curvature vector. 

A consequence of Huisken's monotonicity formula \cite{huisken90} is that the tangent flow at any singular point of a mean curvature flow is modelled by a critical point of the $F$-functional; these critical points are referred to as self-shrinkers, and they are also critical points of the entropy functional. Because they model the singularities in this blow-up sense, the study of self-shrinkers is essential to understanding the mean curvature flow. A further consequence of the monotonicity formula is that entropy is non-increasing under mean curvature flow; as such, the entropy may be interpreted as a useful measure of the complexity of a hypersurface. 

Indeed, the Colding-Minicozzi entropy forms a fundamental component of their theory of generic mean curvature flow: In \cite{CMgeneric} they showed that the only complete, smoothly embedded, entropy-stable self-shrinkers are the generalised cylinders $\mathbf{S}^k(\sqrt{2k})\times\mathbf{R}^{n-k}$, so that under suitable conditions other such singularities may be perturbed away. Here $\mathbf{S}^k(r)$ denotes the round $k$-sphere of radius $r$, and we say that a self-shrinker $\Sigma$ is entropy-stable if it is a local minimum for the entropy functional amongst $C^2$ graphs over $\Sigma$. 
The entropy functional has recently been studied by various other authors; see for instance \cite{bernstein2015topological}, \cite{bernstein2015topology}, \cite{BW}, \cite{CIMW}, \cite{ketover2015entropy} and \cite{liu2016index}. It has also been adapted to other geometric flows; see for example \cite{zhangharmonic} and \cite{KSyangmills}. 

In \cite{CIMW}, Colding, Ilmanen, Minicozzi and White conjectured that the entropy of any closed hypersurface should be at least that of the round sphere (see also \cite{CMP}). In this paper we confirm this conjecture for every dimension $n$; specifically we prove the following:

\begin{theorem}
\label{thm:closedentropyintro}
Let $\Sigma$ be a smooth, closed, embedded hypersurface in $\mathbf{R}^{n+1}$. Then we have $\Lambda(\Sigma)\geq \Lambda(\mathbf{S}^n)$, with equality if and only if $\Sigma$ is a round sphere. 
\end{theorem}

Note that for $n=1$ the result follows immediately from the theorems of Gage-Hamilton \cite{gagehamilton} and Grayson \cite{grayson}, which imply that any smooth closed embedded curve shrinks to a round point. Previously Theorem \ref{thm:closedentropyintro} had been established in the cases $2\leq n\leq 6$ by Bernstein and Wang \cite{BW}, using a cleverly constructed weak flow that ensured the extinction time singularity was of a special type. Ketover and Zhou \cite{ketover2015entropy} also gave an independent proof for the $n=2$, non-toric case using min-max theory for the $F$-functional. Our proof of Theorem \ref{thm:closedentropyintro} results from combining the insightful work of Bernstein-Wang together with our classification of entropy-stable \textit{singular} self-shrinkers, which we now describe. 

As critical points of the $F$-functional, self-shrinkers may equivalently be defined by the elliptic equation \begin{equation}\label{eq:self-shrinkereqintro}\vec{H} = -\frac{1}{2}x^\perp,\end{equation} or as the minimal hypersurfaces for the conformal metric $\e^{-\frac{|x|^2}{2n}}\delta_{ij}$ on $\mathbf{R}^{n+1}$. The simplest examples are the generalised cylinders $\mathbf{S}^k(\sqrt{2k})\times \mathbf{R}^{n-k}$ mentioned above. From the self-shrinker equation (\ref{eq:self-shrinkereqintro}) one can see that any minimal cone in $\mathbf{R}^{n+1}$ (with vertex at the origin) is also a self-shrinker, albeit with a nonempty singular set. The precise notion of singular submanifold we will use in this paper is that of an integer rectifiable (integral) varifold; the definitions of the $F$-functional and entropy functional extend in the natural manner to this setting.

The main theorem of this paper is the following extension of Colding-Minicozzi's classification of entropy-stable self-shrinkers \cite{CMgeneric} to the singular setting:

\begin{theorem}
\label{thm:entropyclassintro}
Let $V$ be an $F$-stationary integral $n$-varifold in $\mathbf{R}^{n+1}$, which has orientable regular part and finite entropy, and satisfies the $\alpha$-structural hypothesis for some $\alpha\in(0,\frac{1}{2})$. Suppose that $V$ is not a generalised cylinder $\mathbf{S}^k(\sqrt{2k})\times \mathbf{R}^{n-k}$. Then $V$ is entropy-unstable.

Furthermore, if $V$ also does not split off a line and is not a cone, the unstable variation may be taken to have compact support away from $\sing V$. If $V$ is a stationary cone, the unstable variation may be taken to be a homogenous variation induced by variation of the link away from its singular set. 
\end{theorem}

For precise definitions the reader is referred to Section \ref{sec:notation}; see also Remark \ref{rmk:multiplicity} regarding multiplicities. The $\alpha$-structural hypothesis here allows us to use the regularity theory of Wickramasekera \cite{wick} to control the singular set; it is automatically satisfied for instance if the singular set has vanishing codimension 1 Hausdorff measure. 
One may recall that in \cite[Theorem 0.14]{CMgeneric}, Colding-Minicozzi also considered the varifold setting, but only in dimensions $n\leq 6$ and with the stronger assumption that the singular set has locally finite codimension 2 measure. Under these conditions the regularity theory ensured that the self-shrinker was smooth (or a regular minimal cone), whereas in our general setting we must handle a more significant singular set. Conjecturally, the singular set of any self-shrinker arising from a smooth mean curvature flow has a singular set of codimension at least 3 (see \cite{ilmanen95sing} or \cite{CMP}).

Theorem \ref{thm:closedentropyintro} will follow from the special case of Theorem \ref{thm:entropyclassintro} classifying \textit{compact} entropy-stable self-shrinkers. Similarly we extend the gap theorem of Bernstein-Wang \cite{BW} for the entropy of compact singular self-shrinkers to all dimensions $n\geq 2$, which itself generalised the main theorem of Colding-Ilmanen-Minicozzi-White \cite{CIMW} to the singular setting for $2\leq n \leq 6$. We also extend the result in \cite{BW} on so-called partially collapsed self-shrinkers to all $n\geq 3$. 

Our approach to Theorem \ref{thm:entropyclassintro} mirrors the approach of Colding-Minicozzi \cite{CMgeneric}, but with several key distinctions. To describe these, recall that their proof consists of three main parts, stated broadly as follows:
\begin{enumerate}
\item Entropy-stability implies $F$-stability;
\item $F$-stability implies mean convexity;
\item Mean convexity implies cylindricality. 
\end{enumerate}
A self-shrinker is $F$-stable if it is stable (under compactly supported variations) for the $F$-functional after ``modding out" by translations and dilations, which turn out to be linearly unstable directions for the $F$-functional on every self-shrinker. 

Point (1) above holds for self-shrinkers that are not invariant under one of these elementary symmetries, that is, for self-shrinkers that do not split off a line and are not cones. Of course, a smooth self-shrinker cannot be a cone, but in the singular case one must account for minimal cones, which are always $F$-unstable yet always entropy-stable under compactly supported variations. To deal with this issue, we introduce the concept of homogenous $F$-stability for minimal cones in terms of the corresponding functional on the links; this concept indeed turns out to be equivalent to entropy-stability under homogenous variations.

We further show that any non-flat minimal cone, which satisfies the $\alpha$-structural hypothesis, is in fact homogenously $F$-unstable and hence entropy-unstable. To do so we need to establish variation formulae for the Gaussian area as functionals on the link. Given these, the key observation is that the link, as a (singular) minimal hypersurface in $S^n$, is quite unstable in the sense that the first eigenvalue $\kappa_1$ of the Jacobi operator is very negative. We have proven such an estimate in \cite{zhu}, which we restate as follows for convenience:

\begin{theorem}[\cite{zhu}]
\label{thm:stabilityeigsphintro}
Let $W$ be a stationary integral $(n-1)$-varifold in $\mathbf{S}^n$ which has orientable regular part and satisfies the $\alpha$-structural hypothesis for some $\alpha\in(0,\frac{1}{2})$. Further suppose that $W$ is not totally geodesic in $\mathbf{S}^n$. Then $\kappa_1(W)\leq -2(n-1)$, with equality if and only if $\supp W$ is a Clifford hypersurface $\mathbf{S}^k\left(\frac{k}{n-1}\right)\times \mathbf{S}^l\left(\frac{l}{n-1}\right)$, where $k+l=n-1$. 
\end{theorem}

Note that in the smooth setting the above estimate is a classical result of J. Simons \cite{simonsjim}.

The most difficult is point (2), for which we show that $F$-stability implies mean convexity on the regular part. As in \cite{CMgeneric}, the key observation is that, on any self-shrinker, the mean curvature $H$ is an eigenfunction of the stability operator $L$ for the $F$-functional, with eigenvalue $-1$. Thus the goal will be to construct $F$-unstable variations when the first eigenvalue of $L$ is less than $-1$, but the singular set causes technical difficulties in the analysis. The main technical obstacle is to obtain effective $L^2$ estimates close to the singular set for the gradient $\nabla \log u$, where $u$ is a positive eigenfunction of $L$ on a subdomain of the regular part, as well as for the second fundamental form $A$. These estimates can be accomplished using a good choice of cutoff functions so long as the singular set has vanishing codimension 4 measure (see Lemma \ref{lem:AL2grad}). This $L^2$ control then allows us to quantify the almost-orthogonality between $H$ and the first eigenfunction $u$ when the subdomain is large enough, which leads to the desired $F$-instability (see Proposition \ref{prop:unstable}).

To complete the classification of singular entropy-stable self-shrinkers, we then extend point (3), the classification of mean convex self-shrinkers, to the singular setting as follows:

\begin{theorem}
\label{thm:meanconvexintro}
Let $V$ be an $F$-stationary integral $n$-varifold in $\mathbf{R}^{n+1}$, with orientable regular part and finite entropy. Further suppose that $\mathcal{H}^{n-1}(\sing V)=0$. If $H\geq 0$ on $\reg V$ then either $V$ is a stationary cone, or $\supp V$ is a generalised cylinder $\mathbf{S}^k(\sqrt{2k})\times\mathbf{R}^{n-k}$. 
\end{theorem}

As in \cite{CMgeneric}, the essential observation for classifying mean convex self-shrinkers is that both $A$ and $H$ are eigentensors of the stability operator, and the key point is to obtain $L^4$ control on $|A|$ in order to justify its use as a test function. To resolve this issue we adapt the Schoen-Simon-Yau \cite{SSY} technique to upgrade the $L^2$ estimates for $|A|$ to the desired $L^4$ bound. 

Let us now briefly outline the structure of this paper. Section \ref{sec:prelims} contains precise definitions as well as background related to entropy and self-shrinkers. Before proving the classification of entropy-stable singular self-shrinkers, we present its applications in Section \ref{sec:applications}. In particular, we describe the proof of Theorem \ref{thm:closedentropyintro} assuming Theorem \ref{thm:entropyclassintro}. 

We then quickly review the Colding-Minicozzi theory in Section \ref{sec:CMtheory}, including the relation between $F$-stability and entropy-stability, and the regularity for stable self-shrinkers. We analyse the Gaussian areas of a cone in Section \ref{sec:Fcone}, treating them as functionals on the link, and (homogenous) $F$-stability of the cone to the stability spectrum of the link.

We fix certain cutoff functions in Section \ref{sec:intsing} which will be used in the remainder of the paper to handle integration around the singular set. In Section \ref{sec:stability}, we show that $F$-stability implies mean convexity, by characterising the bottom of the stability spectrum for singular self-shrinkers and constructing $F$-unstable variations when the first stability eigenvalue $\lambda_1$ is less than $-1$. Then, in Section \ref{sec:meanconvex}, we prove the classification Theorem \ref{thm:meanconvexintro} of mean convex singular self-shrinkers. 

Finally, in Section \ref{sec:class}, we combine our results in order to classify $F$-stable self-shrinkers, homogenously $F$-stable stationary cones and entropy-stable self-shrinkers, in the singular setting. In particular, that section contains the proof of Theorem \ref{thm:entropyclassintro}. 

\begin{acknowledgements}
The author would like to thank Prof. Bill Minicozzi for his warm encouragement and invaluable advice. We would also like to thank Prof. Shing-Tung Yau for his continued support.
Finally we thank Xin Zhou for several insightful conversations. This work is supported in part by the National Science Foundation under grant DMS-1308244.%%%GRANT NUMBER
\end{acknowledgements}

\section{Preliminaries}
\label{sec:prelims}

\subsection{Notation and background}
\label{sec:notation}

\subsubsection{Hypersurfaces}

In this paper a hypersurface will always mean a $C^2$ embedded codimension 1 submanifold $\Sigma$ in a smooth Riemannian manifold $N$. We write $\ol{\nabla}$ for the ambient connection, reserving $\nabla$ for the tangential component and $D$ for the Euclidean connection. Our convention for the Laplacian is \begin{equation} \Lap_\Sigma f = \div_\Sigma(\nabla^\Sigma f).\end{equation} A hypersurface $\Sigma$ is said to be minimal if its mean curvature vector in $N$ is zero, and $\Sigma^n\subset\mathbf{R}^{n+1}$ is said to be a self-shrinker if its mean curvature vector satisfies \begin{equation} \vec{H} = -\frac{1}{2}x^\perp,\end{equation} where $x$ is the position vector. We typically write $y^T$ for the tangential projection of a vector $y\in\mathbf{R}^{n+1}$ and $y^\perp=y-y^T$ for its normal projection.

If $\Sigma$ is two-sided, there is a well-defined unit normal field $\nu$ and we denote by $A$ the second fundamental form of $\Sigma$ along $\nu$. We take the mean curvature on $\Sigma$ to be \begin{equation}H = \div_\Sigma \nu.\end{equation} Note that if the ambient space is orientable, then the hypersurface $\Sigma$ is two-sided if and only if it is orientable (see for instance \cite[Chapter 4]{hirsch}). 

We will typically use $\Sigma^n$ for a hypersurface in $\mathbf{R}^{n+1}$ and $M^{n-1}$ for a hypersurface in $\mathbf{S}^n$, where $\mathbf{S}^n$ denotes the round unit sphere in $\mathbf{R}^{n+1}$. We also denote by $\mathbf{S}^n(r)$ the round sphere of radius $r$. For clarity we will use tildes to distinguish geometric quantities on $M$ from those on $\Sigma$, for instance $\wt{A}$, $\wt{H}$, etc. 

We say that a hypersurface $\Sigma^n\subset\mathbf{R}^{n+1}$ has Euclidean volume growth if there exists a constant $C_V>0$ so that $\vol(\Sigma\cap B_r(x)) \leq C_V r^n$ for any $r>0$ and any $x\in\mathbf{R}^{n+1}$. Here, and henceforth, $B_r(x)$ denotes the open Euclidean ball of radius $r$ in $\mathbf{R}^{n+1}$ centred at $x$. For convenience we set $B_r=B_r(0)$. 

\subsubsection{Varifolds}

In this paper a varifold will always mean an integer rectifiable (integral) varifold $V$ in a Riemannian manifold $N^{n+1}$. We write $\mathcal{H}^k$ for the $k$-dimensional Hausdorff measure in $N$. The reader is directed to \cite{simon} for the basic definitions for varifolds. An integral varifold $V$ is determined by its mass measure, which we denote $\mu_V$. We will always assume that the support $\supp V :=\supp\mu_V$ is connected. We define the regular part $\reg V$ to be the set of points $x \in \supp V$ around which $\supp V$ is locally a $C^2$ hypersurface; the singular set is then $\sing V = \supp V \setminus \reg V$. 

An integer rectifiable $k$-varifold $V$ has an approximate tangent plane $T_x V$ at $\mu_V$-almost every $x$ in $\supp V$. We may thus define the divergence almost everywhere by \begin{equation} (\div_V X) (x)= \div_{T_x V} X(x) = \sum_{i=1}^k \langle E_i,\ol{\nabla}_{E_i} X\rangle(x)\end{equation} where $E_i$ is an orthonormal basis for $T_x V$ and $\ol{\nabla}$ is the ambient connection. The varifold $V$ is then said to have generalised mean curvature vector $\vec{H}$, if $\vec{H}$ is locally integrable and the first variation is given by
\begin{equation} 
\label{eq:genmeancurv}
\int \div_V X \d\mu_V  = -\int \langle X,\vec{H}\rangle \d \mu_V
\end{equation}
for any ambient $C^1$ vector field $X$ with compact support. 

For convenience will say that a varifold $V$ is orientable if and only if $\reg V$ is orientable.

For most of our results we will need some control on the singular set, although we do not assume any such control for now. The weakest condition we will use is the $\alpha$-structural hypothesis of Wickramasekera (\cite{wick}, see also \cite[Section 12]{CMgeneric}): An integral varifold $V$ satisfies the $\alpha$-structural hypothesis for some $\alpha\in(0,1)$, if no point of $\sing V$ has a neighbourhood in which $\supp V$ corresponds to the union of at least three embedded $C^{1,\alpha}$ hypersurfaces with boundary that meet only along their common $C^{1,\alpha}$ boundary. Note that the $\alpha$-structural hypothesis is automatically satisfied if, for instance, $\sing V$ has vanishing codimension 1 Hausdorff measure. 

Note that any hypersurface $\Sigma^n$ with locally bounded $n$-dimensional Hausdorff measure defines an integral varifold that we denote by $[\Sigma]$. 

We say that a $k$-varifold $V$ in $\mathbf{R}^{n+1}$ has Euclidean volume growth if there exists a constant $C_V>0$ so that $\mu_V(B_r(x))  \leq C_V r^k$ for any $r>0$ and any $x\in\mathbf{R}^{n+1}$. 

We will typically use $V$ to denote an integral $n$-varifold in $\mathbf{R}^{n+1}$ and $W$ to denote an integral $(n-1)$-varifold in $\mathbf{S}^n \subset \mathbf{R}^{n+1}$. 

We will say that $V$ splits off a line if it is invariant under translations in some direction; if this is the case then, up to a rotation of $\mathbf{R}^{n+1}$, we may write $\mu_V = \mu_\mathbf{R} \times \mu_{\wt{V}}$ as the product of a multiplicity one line with an integer rectifiable $(n-1)$-varifold $\wt{V}$ in $\mathbf{R}^n$. We say that an integral varifold $V$ is a cone if it is invariant under dilations about the origin; if this is the case then the link $W=V \mr \mathbf{S}^n$ % slicing theorem
is indeed an integer rectifiable $(n-1)$-varifold in $\mathbf{S}^n$ and we write $V=C(W)$. Of course, $C(W)$ is orientable if and only if $W$ is orientable. 

\subsubsection{Gaussian area and entropy}

We denote $\rho_{x_0,t_0}(x)=(4\pi t_0)^{-n/2}\e^{-\frac{|x-x_0|^2}{4t_0}}$. The Gaussian area of $V$ centred at $x_0\in\mathbf{R}^{n+1}$ with scale $t_0>0$ is then given by \begin{equation}F_{x_0,t_0}(V)=\int \rho_{x_0,t_0}\d\mu_V .\end{equation} The normalisation is so that any multiplicity 1 hyperplane has Gaussian area $F_{x_0,t_0}(\mathbf{R}^n)=1$. For convenience we set $\rho=\rho_{0,1}$ and $F=F_{0,1}$. The entropy introduced by Colding-Minicozzi \cite{CMgeneric} may be defined as the supremum over all centres and scales, \begin{equation}\Lambda(V) = \sup_{\substack{x_0\in\mathbf{R}^{n+1}\\t_0>0}} F_{x_0,t_0}(V).\end{equation} 

Note that finite entropy implies Euclidean volume growth:

\begin{lemma} 
\label{lem:volumegrowth}
Let $V$ be an integral $n$-varifold in $\mathbf{R}^{n+1}$ with finite entropy $\Lambda(V)<\infty$. Then for any $x_0$ and any $r>0$, we have \begin{equation} \mu_V(B_r(x_0)) \leq \e^{\frac{1}{4}} (4\pi)^\frac{n}{2} \Lambda(V)r^n.\end{equation}
\end{lemma}
\begin{proof}
As $\e^{-\frac{|x-x_0|^2}{4r^2}} \geq \e^{-\frac{1}{4}}$ for any $x\in B_r(x_0)$, we have \begin{equation} \mu_V(B_r(x_0))\leq \e^{\frac{1}{4}} \int \e^{-\frac{|x-x_0|^2}{4r^2}} \d\mu_V(x).\end{equation} The result follows by definition of the entropy $\Lambda(V)$.
\end{proof}

\subsubsection{Stationary and $F$-stationary varifolds}

We say that an $n$-varifold $V$ in $\mathbf{R}^{n+1}$ is stationary (for area) if it has zero generalised mean curvature $\vec{H}=0$. In particular the regular part must be minimal in $\mathbf{R}^{n+1}$. It is straightforward to see that a cone $V=C(W)$ is stationary if and only if the link $W$ is stationary in $\mathbf{S}^n$. Here an integral $(n-1)$-varifold $W$ in $\mathbf{S}^n \subset \mathbf{R}^{n+1}$ is stationary if its generalised mean curvature in $\mathbf{R}^{n+1}$ is given by $\vec{H}(p)=-(n-1)p$. In particular its regular part is minimal in $\mathbf{S}^n$. 

We say that $V$ is $F$-stationary if it is instead stationary for the $F$-functional defined above, or alternatively with respect to the conformal metric $\e^{-\frac{|x|^2}{2n}}\delta_{ij}$ on $\mathbf{R}^{n+1}$. Equivalently, its generalised mean curvature is given as before by \begin{equation}\vec{H}=-\frac{1}{2} x^\perp.\end{equation} In particular the regular part must be a self-shrinking hypersurface. Also, it follows that a cone $V=C(W)$ is $F$-stationary if and only if it is stationary in $\mathbf{R}^{n+1}$.  

A consequence of Brakke's regularity theorem is that any self-shrinker with entropy close enough to 1 must be a hyperplane \cite{CIMW}. 

Note that the constancy theorem implies that any stationary (or $F$-stationary) varifold has locally constant multiplicity on its regular part. 

\subsubsection{Connectedness}

It will be useful to record a connectedness lemma that follows from the varifold maximum principle of Wickramasekera \cite[Theorem 19.1]{wick} together with the work of Ilmanen in \cite{ilmanen96max}, who proved the same result but with a stronger hypothesis on the singular set.  

\begin{lemma}
\label{lem:connectedness}
Let $V$ be a stationary integral $n$-varifold in a smooth Riemannian manifold $N^{n+1}$, with $\mathcal{H}^{n-1}(\sing V)=0$. Then $\reg V$ is connected if and only if $\supp V$ is connected. 
\end{lemma}

The key point, arguing as in the proof of \cite[Theorem A(ii)]{ilmanen96max}, is that under the assumption $\mathcal{H}^{n-1}(\sing V)=0$, stationarity is equivalent to having vanishing mean curvature together with a local Euclidean volume bound. Each component of $\reg V$ therefore defines a stationary varifold and the varifold maximum principle applies to show that they must coincide. 

\subsubsection{Entropy-stability}

A smooth self-shrinker $\Sigma$ is entropy-stable if it is a local minimum for the entropy functional amongst $C^2$ graphs over $\Sigma$. Here we make this notion precise for varifolds. We first define normal variations that are not required to be compactly supported.

\begin{definition}
\label{def:normalvar}
Let $V$ be an integral $n$-varifold in a manifold $N^{n+1}$ and consider a complete Lipschitz vector field $X$ on $N$. Further suppose that $X$ vanishes on $\sing V$ and is $C^2$ on $N\setminus \sing V$. Writing $\{\Phi^X_s\}_{s\in(-\epsilon,\epsilon)}$ for the flow of $X$, we say that the image varifolds 
\begin{equation} V_s := (\Phi^X_s)_\# V\end{equation}
form a normal variation of $V$ if additionally $X(x) \perp T_x\reg V_s$ for all $s$ and any $x\in \reg V_s$. 
\end{definition}

This definition includes deformations by compactly supported normal graphs over an orientable regular part $\reg V$, since we can construct a smooth ambient field $X$ by extending in a neighbourhood of $\reg V$ away from the singular set. Similarly it includes homogenous variations of a cone $V=C(W)$ in $\mathbf{R}^{n+1}$ induced by compactly supported normal graphs over $\reg W$; in this case the ambient field $X$ only fails to be smooth at the origin.

\begin{definition}
\label{def:entstable1}
We say that an $F$-stationary varifold $V$ is entropy-unstable if there exists a normal variation $V_s$ of $V$ satisfying $\Lambda(V_s)<\Lambda(V)$ for $s>0$. We say that $V$ is entropy-stable if it is not entropy-unstable.
\end{definition}

\subsubsection{$F$-stability}

The notion of $F$-stability of \cite{CMgeneric} extends to the singular setting by requiring that the variation take place away from the singular set.

\begin{definition}
Let $V$ be an orientable $F$-stationary $n$-varifold in $\mathbf{R}^{n+1}$. We say that $V$ is $F$-unstable if there is a normal variation $V_s$ of $V$, compactly supported away from $\sing V$, such that for any variations $x_s$ of $x_0=0$ and $t_s$ of $t_0=1$, we have $\pr_{s}^2|_{s=0} F_{x_s,t_s}(V_s) <0$. 
\end{definition}

$F$-stability is no longer suited for studying the entropy when $V$ is a cone, since one may always zoom away from the compact variation. Therefore, we instead consider homogenous variations and introduce the notion of homogenous $F$-stability for stationary cones as follows:

\begin{definition}
Let $W$ be a stationary $(n-1)$-varifold in $\mathbf{S}^n$. We say that $W$ is homogenously $F$-unstable if there is a normal variation $W_s$ of $W$ in $\mathbf{S}^n$, compactly supported away from $\sing W$, such that for any variation $x_s$ of $x_0=0$, we have $\pr_{s}^2|_{s=0} F_{x_s,1}(C(W_s)) < 0$. We say that $W$ is homogenously $F$-stable if it is not homogenously $F$-unstable.

If $V=C(W)$ is a stationary $n$-cone in $\mathbf{R}^{n+1}$ we say that $V$ is homogenously $F$-stable if and only if $W$ is homogenously $F$-stable. 
\end{definition}

The restriction $t_0=1$ will suffice since for any cone we have $F_{x_0,t_0}(C(W)) = F_{\frac{x_0}{\sqrt{t_0}},1}(C(W))$ by dilation invariance. Note that any stationary cone has finite entropy (see Corollary \ref{cor:coneentropy}). 

\subsubsection{Stability eigenvalues}

Let $\Sigma^n$ be an orientable self-shrinker in $\mathbf{R}^{n+1}$. The second variation operator for the $F$-functional is given by the operator \begin{equation}L=\mathcal{L} + \frac{1}{2} + |A|^2,\end{equation} where $\mathcal{L}$ is the drift Laplacian \begin{equation}\mathcal{L}=\Lap_\Sigma - \frac{1}{2}\langle x,\nabla^\Sigma\cdot\rangle.\end{equation} For (connected) open domains $\Omega \ssubset  \Sigma$ the Dirichlet spectrum $\{\lambda_i(\Omega)\}_{i\geq 1}$ of $L$ on $\Omega$ is well-defined. We define the first stability eigenvalue (with respect to Gaussian area) of $\Sigma$ by 
\begin{equation}
\lambda_1(\Sigma) = \inf_{\Omega} \lambda_1(\Omega) = \inf_f \frac{\int_\Sigma (|\nabla^\Sigma f|^2-|A|^2f^2-\frac{1}{2}f^2)\rho}{\int_\Sigma f^2\rho}.\end{equation}
 Here the second infimum may be taken over Lipschitz functions $f$ with compact support in $ \Sigma$. Note that the infimum could be $-\infty$. If, however, we have $\lambda_1 = \lambda_1(\Sigma)>-\infty$ then we immediately get the stability inequality 
\begin{equation}
\label{eq:stabilityineq}
\int_\Sigma |A|^2f^2\rho \leq \int_\Sigma |\nabla f|^2\rho  + (-\frac{1}{2} - \lambda_1)\int_\Sigma f^2\rho,
\end{equation}
for Lipschitz functions $f$ compactly supported in $ \Sigma$. 
 %could note \lambda_1 <= -1/2
 
 If $V$ is an orientable $F$-stationary varifold, we set $\lambda_1(V) = \lambda_1(\reg V)$. 
 
 It will be useful to recall the following elementary eigenfunctions of $L$:
 
 \begin{lemma}[\cite{CMgeneric}, Theorem 5.2]
On any smooth orientable self-shrinker, for any constant vector $y$ we have $L\langle y,\nu\rangle = \frac{1}{2}\langle y,\nu\rangle$ and $LH=H$. 
 \end{lemma}

For hypersurfaces $M^{n-1}$ in $\mathbf{S}^n$, we will consider the usual stability operator for area given by (recalling that $\mathbf{S}^n$ has constant Ricci curvature $n-1$) \begin{equation} \wt{L} = \Lap_M + |\wt{A}|^2 + (n-1). \end{equation} Here $\wt{A}$ is the second fundamental form of $M$ in $\mathbf{S}^n$. As above we have the Dirichlet spectrum $\{\kappa_i(\Omega)\}_{i\geq 1}$ for any domain $\Omega \ssubset  M$, and we define the first stability eigenvalue of $M$ by \begin{equation}\kappa_1(M) = \inf_{\Omega} \kappa_1(\Omega) = \inf_f \frac{\int_M \left(|\nabla^M f|^2-|\wt{A}|^2f^2-(n-1) f^2\right)}{\int_M f^2}.\end{equation}
Again the infimum may be taken over Lipschitz functions $f$ with compact support in $M$, although it could again be $-\infty$. If $W$ is an orientable stationary integral $(n-1)$-varifold in $\mathbf{S}^n$, we set $\kappa_1(W)=\kappa_1(\reg W)$. 

\subsection{Entropy of $F$-stationary varifolds}

Colding-Minicozzi showed that the entropy of a smooth self-shrinker is achieved by the $F=F_{0,1}$ functional. Ketover and Zhou \cite[Lemma 10.4]{ketover2015entropy} extended their computation to the singular setting (in fact for more general varifolds): 

\begin{lemma}[\cite{ketover2015entropy}]
\label{lem:entropymax}
Let $V$ be an $F$-stationary varifold satisfying $F(V)<\infty$. Fix $a\in\mathbf{R}$ and $y\in\mathbf{R}^{n+1}$ and set $g(s) = F_{sy,1+as^2}(V)$. Then for all $s>0$ with $1+as^2>0$ we have \begin{equation}g'(s) = -\frac{1}{2(1+as^2)}\int \frac{|(asx+y)^\perp|^2s}{1+as^2}\rho_{sy,1+as^2} \d\mu_V(x) \leq 0.\end{equation} Consequently the map $(x_0,t_0)\mapsto F_{x_0,t_0}(V)$ achieves its global maximum at $(0,1)$, that is, $\Lambda(V)=F(V)<\infty$. 
\end{lemma}

As a result we see that stationary cones have finite entropy:

\begin{corollary}
\label{cor:coneentropy}
Let $V=C(W)$ be a stationary $n$-cone. Then $V$ has finite entropy given by $\Lambda(V) = \frac{\|W\|}{\vol(\mathbf{S}^{n-1})},$ where $\mathbf{S}^{n-1}$ is the totally geodesic equator of $\mathbf{S}^n$ and $\|W\|$ is the total mass of the link $W$. 
\end{corollary}
\begin{proof}
A straightforward calculation in polar coordinates gives that $F_{0,1}(V)=\frac{\|W\|}{\vol(\mathbf{S}^{n-1})}<\infty$, and the result then follows from Lemma \ref{lem:entropymax}. 
\end{proof}

We can characterise the equality case in Lemma \ref{lem:entropymax} as follows:

\begin{lemma}
\label{lem:simon}Let $V$ be an $F$-stationary varifold in $\mathbf{R}^{n+1}$. 
\begin{enumerate}
\item If $x^\perp=0$ a.e. on $V$ where $x$ is the position vector, then $\Sigma$ is a stationary cone. 
\item If $y^\perp=0$ a.e. on $V$ for some fixed vector $y$, then $\Sigma$ splits off a line. 
\end{enumerate}
\end{lemma}
\begin{proof}
For point (1) suppose $x^\perp=0$ a.e. on $V$. Then the generalised mean curvature of $V$ is $\vec{H} = -\frac{1}{2}x^\perp =0$, so $V$ is stationary for the area functional. The fact that $V$ is now a cone follows from the monotonicity formula as detailed in the proof of \cite[Theorem 19.3]{simon}. We will not reproduce it here as it is similar to the proof of the second case to follow.

For point (2) suppose $y^\perp=0$ a.e. on $V$. Without loss of generality we may assume $y= e_{n+1}$. We therefore write $x=(x',x_{n+1})$, where $x'\in\mathbf{R}^n$. By the slicing theorem, the slices $V\mr \{x_{n+1} =s\}$ are integral $(n-1)$-varifolds for almost every $s\in\mathbf{R}$. 

Let $\phi:\mathbf{R}\rightarrow \mathbf{R}$ and $f:\mathbf{R}^n\rightarrow \mathbf{R}$ be $C^1$, compactly supported functions. We set \begin{equation}g(s) = \int f(x') \phi( x_{n+1} +s) \d\mu_V(x),\end{equation} so that $g'(s) = \int f(x') \phi'(x_{n+1} +s)\d\mu_V(x).$ 

Consider the vector field $X= f(x')\phi(x_{n+1}+s) e_{n+1}.$ We calculate \begin{equation}\div_V X = \phi(x_{n+1}+s)\langle \nabla f, e_{n+1}\rangle + f(x) \phi'(x_{n+1}+s) \langle e_{n+1}^T, e_{n+1}\rangle\end{equation} Since $e_{n+1}=e_{n+1}^T$ a.e. on $\Sigma$, we have that $\langle \nabla f,e_{n+1}\rangle = \langle Df,e_{n+1}\rangle = 0$, $\langle e_{n+1}^T,e_{n+1}\rangle = 1$ and $\langle x^\perp,e_{n+1}\rangle = 0$. Since $\vec{H}=-\frac{1}{2}x^\perp$, plugging into (\ref{eq:genmeancurv}) then gives that $g'(s)\equiv 0$, hence $g(s)$ is constant in $s$. 

Now fix $a>0$. Using $\phi$ to approximate the characteristic function of the interval $[0,a]$, our work above shows that $\int f(x') \chi_{\{s\leq x_{n+1}\leq s+a\}} \d\mu_V(x)$ is constant in $s$, for any compactly supported $C^1$ function $f$ on $\mathbf{R}^n$. Set $V^s = V \mr \{x_{n+1}=s\}$. For almost every $s\in\mathbf{R}$, both slices $V^s$ and $V^{s+a}$ are integer rectifiable, so using the coarea formula and differentiating gives that \begin{equation}\int f(x') \d\mu_{V^s}(x) = \int f(x')\d\mu_{V^{s+a}}(x)\end{equation} for all such $s$. Another application of the coarea formula then gives that $\mu_V$ is invariant under translation by $ae_{n+1}$. Since $a$ was arbitrary, this concludes the proof.  
\end{proof}

%\begin{remark}
%There is a somewhat subtle point concerning whether the support is invariant, or if only the measures are invariant. In the integer rectifiable case, this issue is resolved as the density is positive on the support (see also \cite{simon, Corollary 42.6}).
%\end{remark}

In particular, the map $(x_0,t_0)\mapsto F_{x_0,t_0}(V)$ has a strict global maximum at $(0,1)$ for $F$-stationary varifolds $V$ that do not split off a line and are not cones. Similarly the map $x_0 \mapsto F_{x_0,1}(V)$ has a strict global maximum at $x_0=0$ if $V$ does not split off a line. 

\section{Applications}
\label{sec:applications}

Before proving Theorem \ref{thm:entropyclassintro}, we will describe how entropy lower bounds for closed hypersurfaces and for singular self-shrinkers can be deduced from the classification of compact entropy-stable singular self-shrinkers. In particular, we will assume for this section that the following holds:

\begin{proposition}
\label{prop:entstablecompact}
Let $V$ be an $F$-stationary integral $n$-varifold in $\mathbf{R}^{n+1}$, which has orientable regular part of multiplicity 1, finite entropy and $\mathcal{H}^{n-1}(\sing V)=0$. If $V$ is not the round sphere $\mathbf{S}^n(\sqrt{2n})$ then there is an entropy-unstable variation of $V$, which is compactly supported away from $\sing V$. 
\end{proposition}

Clearly Proposition \ref{prop:entstablecompact} is an immediate corollary of Theorem \ref{thm:entropyclassintro} (see also Theorem \ref{thm:entropyclass}), since a compactly supported varifold certainly cannot split off a line or be a cone. The main goal of this section will be to prove Theorem \ref{thm:closedentropyintro} under this assumption. The applications we present here extend the results of Bernstein-Wang to all higher dimensions, and depend crucially on their theory developed in \cite{BW}. 

Let $\Lambda_n=\Lambda(\mathbf{S}^n)$ be the entropy of the round sphere. A direct computation (see \cite{stone}) shows \begin{equation}\label{eq:sphentropies} 2 > \Lambda _1 >\frac{3}{2} > \Lambda_2  > \cdots >\Lambda_n > \cdots >1.\end{equation} 

Similar to \cite{BW} we define $\mathcal{SV}_n$ to be the set of all integral $F$-stationary $n$-varifolds in $\mathbf{R}^{n+1}$ with nonempty support. We denote by $\mathcal{CSV}_n$ the subset of varifolds in $\mathcal{SV}_n$ that have compact support. 
For $\Lambda>0$ we also define $\mathcal{SV}_n(\Lambda)$ to be the subset of varifolds in $\mathcal{SV}_n$ with entropy strictly less than $\Lambda$, and $\mathcal{CSV}_n(\Lambda) = \mathcal{SV}_n(\Lambda)\cap \mathcal{CSV}_n$. 

\subsection{Entropy lower bound for closed hypersurfaces}

In \cite{CIMW}, Colding-Ilmanen-Minicozzi-White showed that the shrinking sphere $\mathbf{S}^n(\sqrt{2n})$ minimises entropy amongst smooth, embedded closed self-shrinkers (in fact, they showed that there is a gap to the next lowest entropy in this class). This led them to conjecture the following:

\begin{conjecture}[\cite{CIMW}]
\label{conj:cimw}
Any smoothly embedded, closed hypersurface $\Sigma^n\subset\mathbf{R}^{n+1}$, $n\leq6$ has entropy $\Lambda(\Sigma)\geq \Lambda_n$. \end{conjecture}
 
 The case $n=1$ is an easy consequence of the Gage-Hamilton-Grayson theorem, which states that any embedded closed curve contracts to a round point. Bernstein and Wang \cite{BW} settled Conjecture \ref{conj:cimw} for $2\leq n\leq 6$ by leveraging their insightful observation that under a carefully chosen weak flow, the final time singularity arising from compact initial data must be collapsed in a certain sense (see \cite[Definition 4.6]{BW} and \cite[Definition 4.9]{BW}). In fact, they were able to prove the entropy bound for objects of weaker regularity, the compact boundary measures defined as follows (see also \cite[Definition 2.10]{BW}):

\begin{definition}
Let $V$ be an integral $n$-varifold in $\mathbf{R}^{n+1}$. We call $V$ a compact boundary measure if there is a bounded open nonempty subset $E\subset \mathbf{R}^{n+1}$ of locally finite perimeter (that is, $\chi_E$ has locally bounded variation) such that $\supp \mu_V=\pr E$ and $\mu_V = |D\chi_E|$. 
\end{definition}

In this subsection we will extend their result \cite[Corollary 6.4]{BW} to all dimensions $n\geq 2$. 
 
We will first need Bernstein-Wang's characterisation of the entropy minimiser in $\mathcal{CSV}_n(\Lambda_n)$:

\begin{lemma}[\cite{BW}, Lemma 6.1]
\label{lem:bwmin}
Let $n\geq2$. If for all $1\leq k\leq n-1$, the set $\mathcal{CSV}_k(\Lambda_n)$ is empty, then either $\mathcal{CSV}_n(\Lambda_n)$ is also empty, or there is a $V \in \mathcal{CSV}_n(\Lambda_n)$ satisfying:
\begin{enumerate}
\item $\Lambda(V) = \inf \{\Lambda(\mu): \mu\in\mathcal{CSV}_n(\Lambda_n)\}$,
\item $V$ is a compact boundary measure,
\item $V$ is entropy stable, 
\item $\sing V$ has Hausdorff dimension at most $n-7$.
\end{enumerate}
\end{lemma}

The following proposition is implicit in the proof of \cite[Corollary 6.4]{BW}:

\begin{proposition}[\cite{BW}]
\label{prop:bwfinal}
Consider $n\geq 2$ and let $V$ be a compact boundary measure in $\mathbf{R}^{n+1}$. If for all $2\leq k\leq n$, the set $\mathcal{CSV}_k(\Lambda_k)$ is empty, then $\Lambda(V)\geq \Lambda_n$. Moreover, if equality holds then, up to translations and dilations, $V$ is an entropy-stable member of $\mathcal{CSV}_n$.  
\end{proposition}

We are now ready to prove the main theorem of this section:
 
\begin{theorem}
\label{thm:csmall}
For all $n\geq 2$, we have $\mathcal{CSV}_n(\Lambda_n)=\emptyset$. \end{theorem}
\begin{proof}
First, any $V \in \mathcal{CSV}_1(3/2)$ must be smooth by \cite[Proposition 4.2]{BW} and hence have entropy at least $\Lambda_1$ by \cite[Theorem 0.7]{CIMW}. So by (\ref{eq:sphentropies}) we have $\mathcal{CSV}_1(\Lambda_n)=\emptyset$ for $n\geq 2$. 

We proceed by induction. By \cite[Proposition 6.2]{BW}, we already have $\mathcal{CSV}_n(\Lambda_n)=\emptyset$ for $2\leq n\leq 6$. Now for general $n\geq 2$, if $\mathcal{CSV}_k(\Lambda_k)=\emptyset$ for all $2\leq k\leq n-1$, then using the above discussion we see that the hypotheses of Lemma \ref{lem:bwmin} are satisfied. Thus, if $\mathcal{CSV}_n(\Lambda_n)$ is nonempty then there is a $V\in\mathcal{CSV}_n(\Lambda_n)$ that is entropy-stable and has singular set of codimension at least 7. Moreover, $V$ is a compact boundary measure so its regular part is orientable (see also \cite[Proposition 4.3]{BW}), and it has multiplicity 1 since it is integral with $\Lambda(\Sigma)<\Lambda_n<2$. But then Proposition \ref{prop:entstablecompact} gives that $V$ must be a round sphere, so in particular $\Lambda(V)=\Lambda_n$ which is a contradiction. 
\end{proof}

\begin{corollary}
\label{cor:closedentropy}
Let $n\geq 2$. Any compact boundary measure $V$ in $\mathbf{R}^{n+1}$ has entropy $\Lambda(V)\geq \Lambda(\mathbf{S}^n)$, with equality if and only if $V$ is a round sphere. 
\end{corollary}
\begin{proof}
The lower bound follows immediately from Theorem \ref{thm:csmall} and Proposition \ref{prop:bwfinal}.

If equality holds then, up to a translation and dilation, $V\in\mathcal{CSV}_n$ and $\Lambda(V)=\Lambda_n <\frac{3}{2}$, so as above $V$ is orientable by \cite[Proposition 4.3]{BW}, and $\mathcal{H}^{n-2}(\sing V)=0$ by \cite[Proposition 4.2]{BW}. Since $V$ must also be entropy-stable, by Proposition \ref{prop:entstablecompact} it must be a round sphere.
\end{proof}

Theorem \ref{thm:closedentropyintro} follows from Corollary \ref{cor:closedentropy} for $n\geq 2$, since any closed hypersurface separates $\mathbf{R}^{n+1}$ and hence defines a compact boundary measure. Again the case $n=1$ follows from the Gage-Hamilton-Grayson theorem \cite{gagehamilton, grayson}. 

\subsection{Gap theorem for compact singular self-shrinkers}

The main theorem of Colding-Ilmanen-Minicozzi-White \cite{CIMW} established that the shrinking sphere had the lowest entropy amongst (smooth) closed self-shrinkers, with a gap to the next lowest. Bernstein-Wang, using their own methods, were able to provide an independent proof of this result that in fact extended it to compact singular self-shrinkers, but only for $2\leq n\leq 6$. In this subsection we will extend their result to all $n\geq 2$. 

We will need the following proposition, which is implicit in the proof of \cite[Corollary 6.5]{BW}:

\begin{proposition}[\cite{BW}]
\label{prop:bwgap}
Let $n\geq 2$. Assume that, for all $2\leq k\leq n$:
\begin{itemize}
\item The set $\mathcal{CSV}_k(\Lambda_k)$ is empty;
\item The only compact boundary measure $V \in \mathcal{CSV}_k$ with $\Lambda(V)=\Lambda_k$ is the shrinking sphere $\mathbf{S}^k(\sqrt{2k})$. 
\end{itemize}
Then there exists $\epsilon_n>0$ such that $ \mathcal{CSV}_n(\Lambda_n+\epsilon_n)$ contains only the shrinking sphere $\mathbf{S}^n(\sqrt{2n})$. 
\end{proposition}

Combining Proposition \ref{prop:bwgap} with Theorem \ref{thm:csmall} and Corollary \ref{cor:closedentropy} then immediately yields our gap theorem for compact singular self-shrinkers in all dimensions $n\geq 2$ as follows:

\begin{corollary}
\label{thm:entropygap}
Let $n\geq 2$. There exists $\epsilon_n>0$ so that $ \mathcal{CSV}_n(\Lambda_n+\epsilon_n)$ contains only the shrinking sphere $\mathbf{S}^n(\sqrt{2n})$. 
\end{corollary}

\subsection{Entropy lower bound for partially collapsed self-shrinkers}
We also generalise the results of Bernstein-Wang for so-called partially collapsed self-shrinkers (see \cite[Definition 6.6]{BW}) to all dimensions $n\geq 3$. The following is implicit in the proof of \cite[Corollary 6.7]{BW}:

\begin{proposition}[\cite{BW}]
\label{prop:bwpartial}
Let $n\geq 3$.  Assume that, for all $2\leq k\leq n-1$:
\begin{itemize}
\item The set $\mathcal{CSV}_k(\Lambda_k)$ is empty;
\item The only compact boundary measure $V \in \mathcal{CSV}_k$ with $\Lambda(V)=\Lambda_k$ is the shrinking sphere $\mathbf{S}^k(\sqrt{2k})$. 
\end{itemize}
Then any partially collapsed $V\in\mathcal{SV}_n$ with noncompact support has entropy $\Lambda(V)\geq \Lambda_{n-1}$, with equality if and only if $V$ is the round cylinder $\mathbf{S}^{n-1}(\sqrt{2(n-1)})\times \mathbf{R}$. 
\end{proposition}

As before, we combine Proposition \ref{prop:bwpartial} with Theorem \ref{thm:csmall} and Corollary \ref{cor:closedentropy} to obtain the lower bound for all $n\geq 3$:

\begin{corollary}
Let $n\geq 3$. Any partially collapsed self-shrinker $V\in\mathcal{SV}_n$ with noncompact support has entropy $\Lambda(V)\geq \Lambda_{n-1}$, with equality if and only if $V$ is the round cylinder $\mathbf{S}^{n-1}(\sqrt{2(n-1)})\times \mathbf{R}$. 
\end{corollary}

\section{Colding-Minicozzi theory}
\label{sec:CMtheory}

In this section we recall some results from \cite{CMgeneric}, which allow us to relate entropy-stability to $F$-stability. We will also need variation formulae for the Gaussian area functionals, as well as the regularity theory for self-shrinkers with $\lambda_1$ bounded from below. The proofs found in \cite{CMgeneric} extend naturally to the varifold setting, so we will state the results in this setting. 

\subsection{Variations}
\label{sec:variations}

Here we record a second variation formula for the Gaussian area of an orientable $F$-stationary varifold $V$ in which the centre of the Gaussian functional may change. Specifically, in this subsection we consider normal variations $V_s$ of $V$, with generator $X$ compactly supported away from $\sing V$. 

If $\reg V$ is orientable, each $\reg V_s$ is still orientable with normal denoted $\nu_s$, and the restriction of $X$ is given by $X|_{\reg V_s} = f_s\nu_s$ for some functions $f_s$ compactly supported in $\reg V_s$. For ease of presentation we will give the formulae using the functions $f_s$ with the understanding that $f_s=0$ off the regular part $\reg V_s$.

\begin{proposition}[Second variation at a critical point]
\label{lem:2ndvarcrit}
Let $V$ be an orientable $F$-stationary $n$-varifold in $\mathbf{R}^{n+1}$ with finite entropy. 
Let $V_s$ be a normal variation of $V$ with variation field $X$, compactly supported away from $\sing V$. Write $X|_{\reg V_s} = f_s\nu_s$, with $f=f_0$. Also let $x_s$ and $t_s$ be variations of $x_0=0$ and $t_0=1$ with $x'_0=y$ and $t'_0=a$. Then $\pr_{s}^2|_{s=0} (F_{x_s,t_s}(V_s))$ is given by
\begin{equation}
\int \left(-fLf+2faH-a^2H^2+f\langle y,\nu\rangle - \frac{|y^\perp|^2}{2}\right)\rho \d\mu_V.
\end{equation}
Here we understand the $H^2$ term via the generalised mean curvature, $H^2 = |\vec{H}|^2  = \frac{1}{4}|x^\perp|^2$. 
\end{proposition}

The point is that the proofs of the first and second variation formulae, \cite[Lemma 3.1]{CMgeneric} and \cite[Theorem 4.1]{CMgeneric} respectively, go through essentially unchanged, since the normal variation $V_s$ takes place away from the singular set $\sing V_s=\sing V$ and the contributions of $x_s$ and $t_s$ just come from differentiating the weight. To specialise to a critical point as in \cite[Theorem 4.14]{CMgeneric}, one needs certain integral identities on self-shrinkers; these can be proven in the varifold setting by applying (\ref{eq:genmeancurv}) to the appropriate (exponentially decaying) vector fields. 

\subsection{Entropy stability and $F$-stability}

In this subsection we continue to consider normal variations $V_s$ of an $F$-stationary varifold $V$. 

First, for normal variations compactly supported away from $\sing V$, the proof of \cite[Theorem 0.15]{CMgeneric} goes through to give:

\begin{theorem}
\label{thm:entropyF} 
Suppose $V$ is an orientable $F$-stationary varifold with finite entropy that does not split off a line and is not a cone. If $V$ is $F$-unstable then it is entropy-unstable, where the unstable variation is compactly supported away from $\sing V$. 
\end{theorem}

For stationary cones $V=C(W)$ we need to consider homogenous variations, induced by a normal variation of $W$ in $\mathbf{S}^n$ supported away from $\sing W$. The following is implicit in the proof of \cite[Theorem 0.14]{CMgeneric}:

\begin{theorem}
\label{thm:entropyFcone}
Let $n\geq 3$. Suppose that $V=C(W)$ is an orientable stationary $n$-cone in $\mathbf{R}^{n+1}$ that does not split off a line. If $W$ is homogenously $F$-unstable then $V$ is entropy-unstable with respect to the induced homogenous variation.
\end{theorem}

\subsection{Regularity of self-shrinkers with stability spectrum bounded below}
\label{sec:regularity}

Here we record a regularity result for $F$-stationary varifolds $V$ with $\lambda_1(V)=\lambda_1(\Sigma)>-\infty$ that satisfy the $\alpha$-structural hypothesis, where $\Sigma=\reg V$.

The content of the following proposition is essentially contained in \cite[Section 12]{CMgeneric} and depends on the regularity theory of Wickramasekera \cite{wick}; it follows from the proof of \cite[Proposition 12.24]{CMgeneric}, noting that the proof of \cite[Lemma 12.7]{CMgeneric} goes through with any lower bound $\lambda_1(\Sigma)>-\infty$ because the $\int_\Sigma \phi^2\rho$ term in the stability inequality (\ref{eq:stabilityineq}) may be estimated on small balls (by the Poincar\'{e} inequality) to be small relative to $\int_\Sigma |\nabla \phi|^2\rho$.

\begin{proposition}
\label{prop:regularity}
Let $V$ be an orientable $F$-stationary $n$-varifold in $\mathbf{R}^{n+1}$ with finite entropy, satisfying the $\alpha$-structural hypothesis for some $\alpha\in(0,\frac{1}{2})$. Suppose that $\lambda_1(V) >-\infty$. Then $V$ corresponds to an embedded, analytic hypersurface away from a closed set of singularities of Hausdorff dimension at most $n-7$ (that is empty if $n\leq 6$ and discrete if $n=7$.)
\end{proposition}

\section{Gaussian area functionals on cones}
\label{sec:Fcone}

In this section we consider integral $(n-1)$-varifolds $W$ in $\mathbf{S}^n$. Specifically, we will study the Gaussian areas of their cones $V=C(W)$, which by dilation invariance satisfy \begin{equation}\label{eq:coneFdilation} F_{x_0,t_0}(C(W)) = F_{\frac{x_0}{\sqrt{t_0}},1}(C(W)).\end{equation} As such, it will often be enough to consider centres $x_0\in\mathbf{R}^{n+1}$, with fixed scale $t_0=1$. The main goal is to provide variation formulae for the Gaussian areas $F_{x_0,t_0}(W)$ by treating them as functionals on the link $W$; note that the formulae in Section \ref{sec:variations} do not apply directly since the variations are noncompact. This will give us the means to determine the homogenous $F$-stability of a stationary cone.

Since our focus is on the link, in this section $y^T$ will refer to the projection to the (approximate) tangent space $T_p W$ at a point $p\in \supp W \subset \mathbf{S}^n$, so that $y\in\mathbf{R}^{n+1}$ decomposes as \begin{equation}\label{eq:ydecomp} y = y^T  + \langle y,p\rangle p + y^\perp.\end{equation} Here $y^\perp$ denotes the component orthogonal to $T_pW$ in $T_p \mathbf{S}^n$, which is equivalent to the component orthogonal to $T_p C(W)$ in $\mathbf{R}^{n+1}$, and is given by $y^\perp = \langle y,\wt{\nu}\rangle \wt{\nu}$ on the regular part.

\begin{lemma}
Let $W$ be an integral $(n-1)$-varifold in $\mathbf{S}^n$, and suppose that the cone $C(W)$ has finite entropy. Then we have 
\begin{equation}
F_{x_0,1}(C(W))= (4\pi)^{-\frac{n}{2}} \e^{-|x_0|^2/4} \int K_{n-1}(\langle p,x_0\rangle) \d\mu_W(p),
\end{equation}
where \begin{equation}K_n(t) = \e^{t^2/4}I_n(t)\end{equation} is the sequence of real analytic functions defined by the recurrence relation
\begin{equation}
I_n(t) = t I_{n-1}(t) + 2(n-1)I_{n-2}(t),
\end{equation}
for $n\geq 2$, and
\begin{equation}
I_0(t) = \sqrt{\pi}(1+\erf(t/2)) , \qquad I_1(t) = tI_0(t) + 2\e^{-t^2/4}.
\end{equation}
\end{lemma}
\begin{proof}
Set $V=C(W)$. Using polar coordinates $r>0$, $p\in \mathbf{S}^n$ for $x=rp\in\mathbf{R}^{n+1}$, we have that 
\begin{equation}
\int_{\mathbf{R}^{n+1}} \e^{-\frac{|x-x_0|^2}{4}} \d\mu_V(x)= \int_{\mathbf{S}^n} \left(\int_0^\infty \e^{-\frac{|rp-x_0|^2}{4}} r^{n-1} \d r\right) \d\mu_W(p).
\end{equation}
Completing the square we have $|rp-x_0|^2 = (r-\langle p,x_0\rangle)^2 + |x_0|^2 - \langle p,x_0\rangle^2$, where we have used that $|p|^2=1$. Setting $t=\langle p,x_0\rangle$, it remains to compute the integrals
\begin{equation}
I_n(t) = \int_0^\infty \e^{-\frac{(r-t)^2}{4}}r^n \d r
\end{equation}
for each $n$. First, for $n=0$ by definition of the error function we have \begin{equation}I_0(t) = \int_{-t}^\infty \e^{-u^2/4}\d u = \sqrt{\pi}(1+\erf(t/2)).\end{equation} 
For $n\geq 1$ we have
\begin{eqnarray}
I_n(t)&=& \int_0^\infty (r-t)\e^{-\frac{(r-t)^2}{4}} r^{n-1}\d r +t \int_0^\infty \e^{-\frac{(r-t)^2}{4}}r^{n-1}\d r \\&=&\nonumber \left. -2r^{n-1} \e^{-\frac{(r-t)^2}{4}} \right|_0^\infty +2(n-1)I_{n-2}(t) + tI_{n-1}(t),
\end{eqnarray}
where we have used integration by parts in the second equality. For $n\geq 2$ the first term vanishes whilst for $n=1$ it evaluates to $2\e^{-t^2/4}$, which gives the result.
\end{proof}

\subsection{Variations}
\label{sec:varcone}

For an integral $(n-1)$-varifold $W$ in $\mathbf{S}^n$, we consider normal variations $W_s$ of $W$ in $\mathbf{S}^n$ generated by smooth, compactly supported vector fields $X$ on $\mathbf{S}^n$, so that $X(p) \perp T_p\reg W_s$ for any $s$ and any $p\in \reg W_s$. If $W$ is orientable, then we will write $X|_{\reg W_s} = \phi_s \wt{\nu}_s$. Recall that $\wt{\nu}$ and $\wt{H}$ denote the normal and mean curvature of a hypersurface $M^{n-1}$ in $\mathbf{S}^n$, respectively.

A direct computation yields the first variation formula for the $F$-functional on cones:

\begin{lemma}[First variation formula]
\label{lem:1stvarcone}
Let $W$ be an orientable integral $(n-1)$-varifold in $\mathbf{S}^n$. Let $W_s$ be a normal variation of $W$ in $\mathbf{S}^{n}$ generated by $X$, compactly supported away from $\sing W$. Write $X|_{\reg W} = \phi_s \wt{\nu}_s$ with $\phi=\phi_0$. If $x_s$ is a variation of $x_0$ with $x'_0=y$, then $\pr_s|_{s=0}(F_{x_s,1}(C(W_s)))$ is given by 
\begin{equation}
\frac{\e^{-|x_0|^2/4}}{(4\pi)^\frac{n}{2}} \int \left(\phi \wt{H} K_{n-1}(t) -\frac{1}{2}\langle x_0,y\rangle K_{n-1}(t) +( \langle y,p\rangle + \langle x_0,\nu\rangle \phi )K_{n-1}'(t)\right) \d\mu_W(p),
\end{equation}
where as before we have written $t=\langle p,x_0\rangle$ for convenience.
\end{lemma}

\begin{lemma}[Second variation formula]
\label{lem:2ndvarcone}
Let $W$ be an orientable integral $(n-1)$-varifold in $\mathbf{S}^n$. Let $W_s$ be a normal variation of $W$ in $\mathbf{S}^{n}$ generated by $X$, compactly supported away from $\sing W$. Write $X|_{\reg W} = \phi_s \wt{\nu}_s$ with $\phi=\phi_0$ and $\phi'=\pr_s|_{s=0}\phi_s$. Also let $x_s$ be a variation of $x_0$ with $x'_0=y$, $x''_0=y'$. Then $\pr_{s}^2|_{s=0}(F_{x_s,1}(C(W_s)))$ is given by 
\begin{eqnarray}
\nonumber \frac{\e^{-\frac{|x_0|^2}{4}}}{(4\pi)^\frac{n}{2}}\int \Big[ &  -(\phi \wt{L} \phi) K_{n-1}(t) -\phi\langle x_0,\nabla \phi\rangle K_{n-1}'(t) +2\langle y,\wt{\nu}\rangle \phi K_{n-1}'(t) - \frac{1}{2}|y|^2 K_{n-1}(t) 
\\\nonumber &+ \left(\phi \wt{H} - \frac{1}{2}\langle x_0,y\rangle\right)^2 K_{n-1}(t) + \left(\langle y,p\rangle + \langle x_0,\wt{\nu}\rangle\phi\right)^2 K_{n-1}''(t) 
\\\nonumber&+ 2\left(\phi \wt{H} - \frac{1}{2}\langle x_0,y\rangle\right)\left(\langle y,p\rangle +\langle x_0,\wt{\nu}\rangle\phi\right) K_{n-1}'(t)
\\\nonumber&+  \phi'\left( \wt{H} K_{n-1}(t) +\langle x_0,\wt{\nu}\rangle  K_{n-1}'(t)\right) \\&-\frac{1}{2}\langle x_0,y'\rangle K_{n-1}(t) +\langle p,y'\rangle K_{n-1}'(t)\Big] \d\mu_W(p),
\end{eqnarray}
where again we have written $t=\langle p,x_0\rangle$ for convenience, and $K_{n-1}',K_{n-1}''$ are just the usual derivatives of the single-variable function $K_{n-1}$ (as opposed to the variational derivative). 
\end{lemma}
\begin{proof}
The proof is a direct calculation by differentiating the first variation formula, using that on $M=\reg W$ we have $\wt{\nu}' = -\nabla \phi$ and that $\wt{H}'$ is given by the Jacobi operator, \begin{equation}\wt{H}'= - \Lap_M\phi -|\wt{A}|^2\phi -(n-1)\phi = -\wt{L}_M \phi\end{equation} for hypersurfaces in $\mathbf{S}^{n}$ (see for instance \cite{huiskenpolden}). 
\end{proof}

We will now specialise to the case of a critical point, but first we need some integral identities for minimal hypersurfaces in $\mathbf{S}^n$. 

\begin{lemma}
\label{lem:sphibp}
If $W$ is a stationary integral $(n-1)$-varifold in $\mathbf{S}^n$ then for any fixed vector $y \in \mathbf{R}^{n+1}$ we have
\begin{equation}
\int \langle y,p\rangle\d\mu_W(p) =0,
\end{equation}
\begin{equation}
 \int |y^T|^2 \d\mu_W = (n-1)\int \langle y,p\rangle^2 \d\mu_W(p).
\end{equation}
\end{lemma}
\begin{proof}
We apply (\ref{eq:genmeancurv}) to certain ambient vector fields $X$, recalling that a stationary varifold in $\mathbf{S}^n$ has generalised mean curvature in $\mathbf{R}^{n+1}$ given by $\vec{H}(p)= -(n-1)p$.

For the first claim, simply take $X=y$, so that $\div_W X=0$. 

For the second claim, take $X= \langle y,x\rangle y$, then we have $\div_W X = \langle y^T,y\rangle = |y^T|^2$ and $\langle p,X(p)\rangle = \langle y,p\rangle^2$. 
\end{proof}

\begin{proposition}[Second variation at a critical point]
\label{lem:2ndvarcritcone}
Let $W$ be an orientable stationary integral $(n-1)$-varifold in $\mathbf{S}^n$. Let $W_s$ be a normal variation of $W$ in $\mathbf{S}^{n}$ generated by $X$, compactly supported away from $\sing W$. Write $X|_{\reg W} = \phi_s \wt{\nu}_s$ with $\phi=\phi_0$. Also let $x_s$ be a variation of $x_0=0$ with $x'_0=y$. Then $\pr_{s}^2|_{s=0} (F_{x_s,1}(C(W_s)))$ is given by
\begin{equation}
\label{eq:2ndvarcritcone}
\frac{1}{2}\pi^{-\frac{n}{2}}  \Gamma\left(\frac{n}{2}\right) \int \left( -\phi \wt{L} \phi +2\frac{\Gamma(\frac{1+n}{2})}{\Gamma(\frac{n}{2})} \phi\langle y,\wt{\nu}\rangle - \frac{1}{2}|y^\perp|^2 \right) \d\mu_W(p)
\end{equation}
\end{proposition}
\begin{proof}
Using the recurrence for $K_{n-1}$ one may verify the special values $K_{n-1}(0)=2^{n-1}\Gamma(\frac{n}{2})$, $K_{n-1}'(0) = 2^{n-1} \Gamma(\frac{1+n}{2})$ and $K''_{n-1}(0)=2^{n-2}n\Gamma(\frac{n}{2})$. Plugging $x_0=0$ and $\wt{H}=0$ into Lemma \ref{lem:2ndvarcone}, we get that $\pr_{s}^2|_{s=0}(F_{x_0,1}(C(W_s)))$ is given by 
\begin{eqnarray}
\frac{1}{2}\pi^{-\frac{n}{2}}\Gamma\left(\frac{n}{2}\right) \int \Big[& -\phi \wt{L} \phi +2\frac{\Gamma(\frac{1+n}{2})}{\Gamma(\frac{n}{2})} \phi\langle y,\wt{\nu}\rangle  -\frac{1}{2}|y|^2 \\& \nonumber+ \frac{n}{2} \langle y,p\rangle^2 + \frac{\Gamma(\frac{1+n}{2})}{\Gamma(\frac{n}{2})} \langle y',p\rangle  \Big] \d\mu_W(p) ,
\end{eqnarray}
where $y'=x''_0$. Using Lemma \ref{lem:sphibp} to handle the last three terms completes the proof, recalling that according to the decomposition (\ref{eq:ydecomp}) we have $|y|^2 = |y^T|^2 + \langle y,p\rangle^2 + |y^\perp|^2$.
\end{proof}

\begin{remark}
\label{rmk:2ndvarcritcone}
If $V=C(W)$ is a stationary cone then, working in polar coordinates $r=|x|$ on the regular part $\Sigma = \reg C(W)$, the stability operator $L_\Sigma$ has the decomposition \begin{equation}
\label{eq:operatordecomp}
Lf = r^{-2}\Lap_M f + \frac{n-1}{r} \pr_r f + \pr_r^2 f - \frac{r}{2}\pr_r f + \frac{|\wt{A}|^2}{r^2}f + \frac{1}{2}f = r^{-2}(\wt{L}_M-(n-1)+L_1)f,
\end{equation}
where \begin{equation}L_1 = r^2\pr_r^2 + (n-1)r\pr_r - \frac{r^3}{2}\pr_r + \frac{r^2}{2}.\end{equation} Noting that $L_1 r = (n-1)r$, and using the evaluation of the special integrals $I_n(0)$, it follows that the integral over the cone $C(W)$
\begin{equation}
\int \left(-f Lf + f\langle y,\nu\rangle - \frac{|y^\perp|^2}{2}\right)\rho \d\mu_{C(W)}
\end{equation}
coincides with (\ref{eq:2ndvarcritcone}) if we set $f(x) = r \phi(\frac{x}{r})$. This shows in particular that the second variation formula Proposition \ref{lem:2ndvarcrit} is valid for homogenous variations of a stationary cone.
\end{remark}

We record the following estimate for the coefficient of the middle term of (\ref{eq:2ndvarcritcone}). 

\begin{lemma}
\label{lem:gamma}
For any integer $n\geq 2$ we have \begin{equation}  \frac{\Gamma(\frac{1+n}{2})^2}{\Gamma(\frac{n}{2})^2} <n-1.\end{equation}
\end{lemma}
\begin{proof}
Let $A_n = \frac{1}{n-1} \frac{\Gamma(\frac{1+n}{2})^2}{\Gamma(\frac{n}{2})^2}$. By the functional equation for the gamma function, we have for all $n>3$ that $A_n = \frac{(n-1)(n-3)}{(n-2)^2} A_{n-2} < A_{n-2}$, so the lemma follows from checking that $A_2 = \frac{\Gamma(3/2)^2}{\Gamma(1)^2} = \frac{\pi}{4}<1$ and $A_3 = \frac{1}{2}\frac{\Gamma(2)^2}{\Gamma(3/2)^2} = \frac{2}{\pi}<1$. 
\end{proof}

\section{Integration on singular hypersurfaces}
\label{sec:intsing}

In this section we present some technical results that will allow us to work on the regular part of an integral varifold with small enough singular set.

\subsection{Cutoff functions}
\label{sec:cutoff}

Given an integral $n$-varifold $V$ in $\mathbf{R}^{n+1}$ satisfying $\mathcal{H}^{n-q}(\sing V)=0$, $q\geq 0$, we describe here our choice of cutoff functions (on $\mathbf{R}^{n+1}$) that will allow us to integrate around the singular set.

For any fixed $R>4$ and $\epsilon>0$, since the singular set is closed, using the definition of Hausdorff measure we may cover the compact set $\sing V \cap B_{R}$ by finitely many Euclidean balls, \begin{equation}\sing V\cap B_{R} \subset \bigcup_{i=1}^m B_{r_i}(p_i), \qquad\text{where} \qquad\sum_i r_i^{n-q} <\epsilon,\end{equation} and of course we may assume without loss of generality that $r_i<1$ for each $i$. This covering depends on $q$, $R$ and $\epsilon$, but we will suppress this dependence in the notation. 

Given such a covering, we may take smooth cutoff functions $0\leq \phi_i\leq 1$ such that $\phi_i = 1$ outside $B_{3r_i}(p_i)$ and $\phi_i=0$ inside $B_{2r_i}(p_i)$, with $|D\phi_i| \leq \frac{2}{r_i}$ in between. We will also need to cut off on large balls so we fix a cutoff function $0\leq \eta_R\leq 1$ such that $\eta_R=1$ inside $B_{R-3}$ and $\eta_R=0$ outside $B_{R-2}$, with $|D \eta_R|\leq 2$ in between. Then, we combine these cutoffs by setting $\phi_{R,\epsilon} = \inf_i ( \phi_i,\eta_R)\leq 1$, which is Lipschitz with compact support in $B_{R-1}\setminus \bigcup_{i=1}^m B_{2r_i}(p_i)$, and satisfies $|D\phi_{R,\epsilon}| \leq \sup_i (|D\phi_i|, |D\eta_R|)$. 

We will also need cutoff functions on annuli by smooth functions $0\leq \psi_i\leq \frac{2}{r_i}$ satisfying $\psi_i = \frac{2}{r_i}$ inside $B_{3r_i}(p_i) \setminus B_{2r_i}(p_i)$ and $\psi_i=0$ outside $B_{4r_i}(p_i)\setminus B_{r_i}(p_i)$, with $|D\psi_i| \leq \frac{4}{r_i^2}$ in between. We also take $0\leq \xi_R\leq 2$ such that $\xi_R =2$ inside $B_{R-2}\setminus B_{R-3}$ and $\xi_R=0$ outside $B_{R-1}\setminus B_{R-4}$, with $|D\xi_R| \leq 4$ in between. We combine these by setting $\psi_{R,\epsilon} = \sup_i(\psi_i,\xi_R)$, which is Lipschitz and satisfies $|D\psi_{R,\epsilon}| \leq \sup_i(|D\psi_i|,|D\xi_R|).$ In particular, we have \begin{equation} |D\phi_{R,\epsilon}| \leq \psi_{R,\epsilon}.\end{equation}

We will reduce the dependence to the single parameter $R$ by choosing $\epsilon = \epsilon(R)$ such that $\lim_{R\rightarrow \infty} \epsilon(R)=0$. In this setting we write more compactly $\phi_R = \phi_{R,\epsilon}$, $\psi_R=\psi_{R,\epsilon}$. 

\subsection{Integration}
\label{sec:integration}

We will conduct our analysis in the weighted $L^p$ spaces introduced in \cite{CMgeneric}. We say that a function $f$ is weighted $L^p$ on a hypersurface $\Sigma$ if it is $L^p$ with respect to the measure $\rho \d\mu_\Sigma$. That is, for $p\in (0,\infty)$ we say $f$ is weighted $L^p$ if $\|f\|_{p}:= \left(\int_\Sigma |f|^p \rho\right)^\frac{1}{p}<\infty$, and for $p=\infty$ we require $\|f\|_\infty =\sup_\Sigma |f|<\infty$. The weighted $W^{k,p}$ spaces are defined analogously. The goal of this subsection is to establish conditions under which integration by parts is justified in these spaces. 

Recall that the operator $\mathcal{L}$ is symmetric with respect to the weight $\rho$:

\begin{lemma}[\cite{CMgeneric}, Lemma 3.8]
\label{lem:compactibp}
If $\Sigma\subset \mathbf{R}^{n+1}$ is any hypersurface, $u$ is a $C^1$ function with compact support in $\Sigma$ and $v$ is a $C^2$ function, then \begin{equation} \int_\Sigma u(\mathcal{L} v) \rho = -\int_\Sigma \langle \nabla v,\nabla u\rangle \rho.\end{equation}
\end{lemma}

In the remainder of this subsection $\Sigma^n$ will denote the regular part of an $n$-varifold $V$ in $\mathbf{R}^{n+1}$ with Euclidean volume growth. The exponential decay of the weight $\rho = (4\pi)^{-n/2} \e^{-|x|^2/4}$ then gives that any function on $\Sigma$ of polynomial growth in $|x|$ is automatically weighted $L^p$ for any $p \in (0,\infty)$. 

\begin{lemma}
\label{lem:gradconv}
Let $q>0$ and suppose that $\mathcal{H}^{n-q}(\sing V)=0$. Let $\Sigma=\reg V$, and take $\phi_R = \phi_{R,\epsilon}$ as in Section \ref{sec:cutoff}. Then we have the following gradient estimate for $\phi_R$: \begin{equation}\int_\Sigma |\nabla \phi_R|^q \rho \leq 2^q C_V(R^n\e^{-\frac{(R-3)^2}{4}} + 3^n\epsilon),\end{equation} where $C_V$ is the volume growth constant. In particular $\lim_{R\rightarrow \infty} \int_\Sigma |\nabla \phi_R|^q \rho =0$. 
\end{lemma} 
\begin{proof}
We have \begin{eqnarray}
\int_\Sigma |\nabla \phi_R|^{q}\rho &\leq&  \int_{\Sigma \cap B_{R-2}\setminus B_{R-3}} 2^{q} \rho  +\sum_{i=1}^m \int_{\Sigma \cap B_{3r_i}(p_i) \setminus B_{2r_i}(p_i)} \frac{2^{q}}{r_i^{q}} \\\nonumber&\leq& 2^qC_V\left( R^n\e^{-\frac{(R-3)^2}{4}}+3^n\sum_{i} r_i^{n-q}\right) \\\nonumber&\leq& 2^q C_V(R^n\e^{-\frac{(R-3)^2}{4}} + 3^n\epsilon).\end{eqnarray} 
The limit follows since we choose $\epsilon$ such that $\lim_{R\rightarrow \infty}\epsilon(R)=0$. 
\end{proof}

\begin{corollary}
\label{cor:fgrad}
Assume $\mathcal{H}^{n-q}(\sing V)=0$ for some $q$. Let $\Sigma=\reg V$ and $\phi_R$ be as above. 
\begin{enumerate}
\item Suppose that $q\geq 1$ and that $f$ is weighted $L^p$, $p = \frac{q}{q-1}$. Then \begin{equation}\lim_{R\rightarrow \infty} \int_\Sigma |f||\nabla \phi_R|\rho = 0.\end{equation} 
\item Suppose that $q\geq 2$ and that $f$ is weighted $L^p$, $p = \frac{2q}{q-2}$. Then \begin{equation}\lim_{R\rightarrow \infty} \int_\Sigma f^2 |\nabla \phi_R|^2\rho = 0.\end{equation} 

\end{enumerate}
\end{corollary}
Note again that here we allow $p=\infty$. 
\begin{proof}
For (1), using H\"{o}lder's inequality, we have
\begin{equation}
\int_\Sigma |f| |\nabla \phi_R| \rho \leq \|f\|_{p}\left( \int_\Sigma |\nabla \phi_R|^{q}\rho\right)^\frac{1}{q}
\end{equation}
where $\frac{1}{p}+\frac{1}{q}=1$. 

Similarly for (2) we have
\begin{equation}
\int_\Sigma |f|^2 |\nabla \phi_R|^2 \rho \leq \|f\|_{p}^2\left( \int_\Sigma |\nabla \phi_R|^{q}\rho\right)^\frac{2}{q}
\end{equation}
where $\frac{2}{p}+\frac{2}{q}=1$. 

By supposition the weighted $L^p$-norms of $f$ are finite, so both results now follow from Lemma \ref{lem:gradconv}. 
\end{proof}

\begin{lemma}
\label{lem:ibp}
Suppose that $\mathcal{H}^{n-q}(\sing V)=0$ for some $q\geq 1 $. Further suppose that $u,v$ are $C^2$ functions on $\Sigma=\reg V$ such that $|\nabla u||\nabla v|$ and $|u\mathcal{L}v|$ are weighted $L^1$, and $|u\nabla v|$ is weighted $L^p$, $p=\frac{q}{q-1}$. Then \begin{equation} 
\int_\Sigma (u\mathcal{L}v) \rho = -\int_\Sigma \langle \nabla u,\nabla v\rangle \rho.\end{equation}
\end{lemma}
\begin{proof}
If $\phi$ has compact support we may use Lemma \ref{lem:compactibp} to get \begin{equation}
\int_\Sigma \phi u(\mathcal{L}v)\rho = -\int_\Sigma\phi \langle \nabla u,\nabla v\rangle \rho - \int_\Sigma   u\langle \nabla v,\nabla \phi\rangle\rho. \end{equation} 

Applying this to $\phi= \phi_R$, Corollary \ref{cor:fgrad} gives that the second term on the right tends to zero as $R\rightarrow\infty$, so the result follows by dominated convergence.

\end{proof}

In practice we will refer to both Lemma \ref{lem:compactibp} and Lemma \ref{lem:ibp} simply as integration by parts.

\section{Stability of singular self-shrinkers}
\label{sec:stability}

Throughout this section $\Sigma^n$ will denote an orientable self-shrinker in $\mathbf{R}^{n+1}$ with Euclidean volume growth $\vol(\Sigma \cap B_r(x)) \leq C_V r^n$. The main goals of this section are to understand the first stability eigenvalue of $\Sigma$ and to construct $F$-unstable variations when it is low enough. 

Frequently we will take $\Sigma$ to be the regular part of an $F$-stationary varifold $V$ with finite entropy (which has Euclidean volume growth by Lemma \ref{lem:volumegrowth}), and the results will depend on the size of the singular set. In several cases the assumptions on $\sing V$ may be weakened using the regularity theory Proposition \ref{prop:regularity}, but we state the stronger hypotheses to clarify the degree of regularity required. 

\subsection{Stability spectrum of $\Sigma$}

Recall that the first stability eigenvalue of the stability operator 
\begin{equation} L = \Lap_\Sigma -\frac{1}{2}\langle x,\nabla^\Sigma \cdot\rangle + |A|^2+\frac{1}{2}\end{equation} on a self-shrinker $\Sigma$ is defined by 
\begin{equation}\lambda_1(\Sigma) = \inf_{\Omega} \lambda_1(\Omega) = \inf_f \frac{\int_\Sigma (|\nabla f|^2-|A|^2f^2-\frac{1}{2}f^2)\rho}{\int_\Sigma f^2\rho},\end{equation}
where the infimum is taken over functions compactly supported in $ \Sigma$, and could potentially be $-\infty$. 
Also recall that if indeed $\lambda_1=\lambda_1(\Sigma)>-\infty$, then we have the stability inequality 
\begin{equation}
\int_\Sigma |A|^2f^2\rho \leq \int_\Sigma |\nabla f|^2\rho  + (-\frac{1}{2} - \lambda_1)\int_\Sigma f^2\rho,
\end{equation}
for Lipschitz functions $f$ compactly supported in $ \Sigma$. 

\begin{lemma}
\label{lem:specbotlower}
Suppose that $u>0$ is a $C^2$ function on $\Sigma$ with $Lu=-\lambda u$. Then $\lambda_1(\Sigma)\geq \lambda$. 

Moreover, if $f$ is Lipschitz with compact support in $\Sigma$, then \begin{equation}
\label{eq:eigestimatecpct}
\int_\Omega f^2(|A|^2+|\nabla \log u|^2)\rho \leq \int_\Omega (4|\nabla f|^2-2\lambda f^2)\rho.
\end{equation}
\end{lemma}
\begin{proof}
Since $u>0$, the function $\log u$ is well-defined on $\Sigma$ and we can compute that
\begin{equation}
\mathcal{L}\log u= -\lambda -\frac{1}{2} - |A|^2 - |\nabla \log u|^2.
\end{equation}
Since $f$ has compact support in $\Sigma$, then integrating $f^2\mathcal{L}\log u$ by parts we have that
\begin{equation}
\int_\Sigma \left(\lambda+\frac{1}{2} + |A|^2 + |\nabla \log u|^2\right)f^2\rho = \int_\Sigma \langle \nabla f^2,\nabla \log u\rangle\rho.
\end{equation}
Using the absorbing inequality $|\langle\nabla f^2,\nabla \log u\rangle| \leq |\nabla f|^2 + f^2|\nabla \log u|^2$ we get that
\begin{equation}
\int_\Sigma \left(\lambda+\frac{1}{2} + |A|^2 \right)f^2\rho \leq \int_\Sigma |\nabla f|^2\rho
\end{equation}
and hence \begin{equation}\frac{\int_\Sigma (|\nabla f|^2-|A|^2f^2-\frac{1}{2}f^2)\rho}{\int_\Sigma f^2\rho} \geq \lambda.\end{equation}
Since this holds for any $f$ with compact support in $\Sigma$, we conclude that $\lambda_1(\Sigma)\geq \lambda$ as claimed.

If we instead absorb using $|\langle \nabla f^2,\nabla \log u\rangle| \leq 2|\nabla f|^2 + \frac{1}{2}f^2|\nabla \log u|^2$ we get that \begin{equation} \int_\Sigma \left(\lambda+\frac{1}{2}+|A|^2+\frac{1}{2}|\nabla \log u|^2\right)f^2\rho \leq 2\int_\Sigma |\nabla f|^2\rho,\end{equation} which implies the bound (\ref{eq:eigestimatecpct}).
\end{proof}

We will frequently apply Lemma \ref{lem:specbotlower} to subdomains $\Omega$ of the regular part of an $F$-stationary varifold as well as to the regular part itself.

\subsubsection{Weighted integral estimates}

\begin{lemma}
\label{lem:eigestimate}
Let $V$ be an orientable $F$-stationary $n$-varifold in $\mathbf{R}^{n+1}$ with finite entropy and $\mathcal{H}^{n-q}(\sing V)=0$ for some $q\geq 2$. Suppose that $u>0$ is a $C^2$ function on $\Sigma=\reg V$ with $Lu=-\lambda u$. Then if $\phi$ is weighted $W^{1,2}$ and weighted $L^p$, $p= \frac{2q}{q-2}$. Then
\begin{equation}
\int_\Omega \phi^2(|A|^2+|\nabla \log u|^2)\rho \leq \int_\Omega (8|\nabla \phi|^2-2\lambda\phi^2)\rho.
\end{equation}
\end{lemma}
\begin{proof}
We take $f = \phi_R \phi$, where $\phi_R$ is as in Section \ref{sec:intsing}. Applying Lemma \ref{lem:specbotlower} we get that \begin{equation} \int_\Sigma \phi_R^2 \phi^2(|A|^2+|\nabla \log u|^2)\rho \leq \int_\Sigma (8\phi^2|\nabla \phi_R|^2 + 8\phi_R^2 |\nabla \phi|^2 -2\lambda \phi_R^2 \phi^2)\rho.\end{equation}

As $R\rightarrow \infty$, the second and third terms on the right converge since $\phi$ is weighted $W^{1,2}$, and Corollary \ref{cor:fgrad} implies that the first term on the right term tends to zero, whence Fatou's lemma gives the result.
\end{proof}

For any integer $k\geq 0$, the function $|x|^{2k}$ is a polynomial in $x$, so by the Euclidean volume growth it is of course $W^{1,p}$ for any $p\in(0,\infty)$. Thus we immediately get:

\begin{corollary}
\label{lem:AL2}
Let $V$ be an orientable $F$-stationary $n$-varifold in $\mathbf{R}^{n+1}$ with finite entropy and $\mathcal{H}^{n-q}(\sing V)=0$ for some $q>2$. Suppose that $u>0$ is a $C^2$ function that satisfies $Lu=-\lambda u$ on $\Sigma=\reg V$. Then $|A||x|^k$ and $|x|^k |\nabla \log u|$ are weighted $L^2$ for any $k\geq 0$.\end{corollary}

We now record the main quantitative $L^2$ estimates for $|A|$ and $|\nabla \log u|$ that will be essential both for constructing unstable variations when $\lambda_1<-1$, and for classifying mean convex self-shrinkers. It is crucial that the estimate holds for positive eigenfunctions $u$ defined only on a subdomain $\Omega$. 

\begin{lemma}
\label{lem:AL2grad}
Let $V$ be an orientable $F$-stationary $n$-varifold in $\mathbf{R}^{n+1}$ with finite entropy and $\mathcal{H}^{n-4}(\sing V)=0$. Let $\phi_R= \phi_{R,\epsilon}$ be as in Section \ref{sec:cutoff}, and consider a domain $\Omega$ such that $\supp \phi_R \subset \Omega \subset \Sigma=\reg V$. If $u$ is a positive $C^2$ function on $\Omega$ satisfying $Lu=-\lambda u$, then 
\begin{equation}
\int_\Omega (|A|^2+|\nabla \log u|^2)\phi_R^2 |\nabla \phi_R|^2 \rho \leq (256+8|\lambda|)C_V (R^n \e^{-\frac{(R-4)^2}{4}} + 4^{n}\epsilon).
\end{equation}
\end{lemma}
\begin{proof}
Recall that we cover the singular set $\sing V\cap B_{R} \subset \bigcup_{i=1}^m B_{r_i}(p_i)$, where $\sum_{i=1}^m r_i^{n-4} <\epsilon$ and without loss of generality $r_i<1$ for each $i$. 

The key is to replace $|\nabla \phi_R|$ by the annular bump function $\psi_R=\psi_{R,\epsilon} \geq |\nabla \phi_R|$, which has better regularity properties:
\begin{equation}
\int_\Omega (|A|^2+|\nabla \log u|^2)\phi_R^2 |\nabla \phi_R|^2 \rho \leq \int_\Omega (|A|^2+|\nabla \log u|^2)\phi_R^2\psi_R^2 \rho.
\end{equation}
In particular we may now apply Lemma \ref{lem:specbotlower} to $f = \phi_R\psi_R$ on the hypersurface $\Omega$ to get \begin{equation}
\label{eq:AL2grad}
\int_\Omega (|A|^2+|\nabla \log u|^2)\phi_R^2\psi_R^2 \rho \leq \int_\Omega (8 \psi_R^2 |\nabla \phi_R|^2 + 8 \phi_R^2 |\nabla \psi_R|^2 + 2|\lambda| \phi_R^2\psi_R^2)\rho. 
\end{equation}

We may bound the first term on the right in (\ref{eq:AL2grad}) by 
\begin{eqnarray}
\int_\Sigma \psi_R^2 |\nabla \phi_R|^2\rho &\leq& \int_\Sigma \psi_R^4 \rho \leq \int_{\Sigma \cap B_{R-1}\setminus B_{R-4}} 16 \rho + \sum_{i=1}^m \int_{\Sigma \cap B_{4r_i}(p_i)\setminus B_{r_i}(p_i)} \frac{16}{r_i^4} \\&\leq& \nonumber 16 C_V\left(R^n \e^{-\frac{(R-4)^2}{4}}+4^n \sum_i r_i^{n-4}\right) \\&\nonumber\leq& 16 C_V (R^n \e^{-\frac{(R-4)^2}{4}}+4^n \epsilon).
\end{eqnarray}

Since $\phi_R^2\leq1$ the second term on the right in (\ref{eq:AL2grad}) is bounded by 
\begin{eqnarray}
\int_\Sigma \phi_R^2 |\nabla \psi_R|^2\rho &\leq& \int_{\Sigma \cap B_{R-1}\setminus B_{R-4}} 16 \rho + \sum_{i=1}^m \int_{\Sigma \cap B_{4r_i}(p_i)\setminus B_{r_i}(p_i)} \frac{16}{r_i^4} \\&\leq& \nonumber 16 C_V\left(R^n \e^{-\frac{(R-4)^2}{4}}+4^{n} \sum_i r_i^{n-4}\right) \\&\nonumber\leq& 16 C_V (R^n \e^{-\frac{(R-4)^2}{4}}+4^{n} \epsilon),
\end{eqnarray}
and since $r_i<1$ the last term is bounded by 
\begin{eqnarray}
\int_\Sigma \phi_R^2 \psi_R^2\rho &\leq& \int_{\Sigma \cap B_{R-1}\setminus B_{R-4}} 4\rho + \sum_{i=1}^m \int_{\Sigma \cap B_{4r_i}(p_i)\setminus B_{r_i}(p_i)} \frac{4}{r_i^2} \\&\leq& \nonumber  4C_V\left(R^n \e^{-\frac{(R-4)^2}{4}}+4^{n} \sum_i r_i^{n-2}\right) \\&\nonumber\leq& 4C_V (R^n \e^{-\frac{(R-4)^2}{4}}+4^{n} \epsilon).
\end{eqnarray}
Combining these estimates gives the result as claimed.
\end{proof}

\subsubsection{Bottom of the spectrum}

\begin{lemma}
\label{lem:specbot}
Let $\Sigma^n$ be a connected, orientable self-shrinker with $\lambda_1=\lambda_1(\Sigma) > -\infty$. Then there is a positive $C^2$ function on $\Sigma$ with $Lu=-\lambda_1 u$. 

Moreover, suppose that $\Sigma$ is the regular part of an $F$-stationary $n$-varifold $V$ with finite entropy and $\mathcal{H}^{n-q}(\sing V)=0$ for some $q\geq 2$. If $v$ is a $C^2$ function on $\Sigma$ with $Lv=-\lambda_1v$, which is weighted $W^{1,2}$ and weighted $L^p$, $p=\frac{2q}{q-2}$, then $v=cu$ for some $c\in \mathbf{R}$.
\end{lemma}
\begin{proof}
For the existence of $u$ we proceed as in \cite{CMgeneric}: Fix $p\in\Sigma$ and consider an exhaustion $p\in \Omega_1 \subset \Omega_2 \subset \cdots$ of $ \Sigma = \bigcup_i\Omega_i$. For each $i$ there is a positive Dirichlet eigenfunction $Lu_i = -\lambda_1(\Omega_i)u_i$ on $\Omega_i$, and we may normalise so that $u_i(p)=1$. Since $\lambda_1(\Omega_i)$ decreases monotonically to $\lambda_1>-\infty$, the Harnack inequality gives $1\leq \sup u_i \leq C \inf u_i \leq C$, where $C=C(\Omega_i,\lambda_1)$. Elliptic theory gives uniform $C^{2,\alpha}$ bounds on the $u_i$ on each compact set, so we get a subsequence converging uniformly in $C^2$ to a nonnegative solution of $Lu=-\lambda_1 u$ on $ \Sigma$ with $u(p)=1$. The Harnack inequality again implies that $u$ is positive on $\Sigma$.

For the uniqueness, by the assumptions on $v$, Lemma \ref{lem:eigestimate} gives that $|A|v$ and $v|\nabla \log u|$ are weighted $L^2$. By expansion this implies that $v\mathcal{L}v$ and $v^2\mathcal{L}\log u$ are weighted $L^1$, and since $v$ is weighted $W^{1,2}$ we see that $|\nabla v^2| |\nabla \log u| \leq |\nabla v|^2+ v^2|\nabla \log u|^2$ is weighted $L^1$. Moreover since $\frac{1}{2}+\frac{1}{p} = \frac{q-1}{q}$, H\"{o}lder's inequality gives that $\|v\nabla v\|_{\frac{q}{q-1}}\leq \|\nabla v\|_2 \|v\|_p <\infty$ and $\|v^2\nabla \log u\|_{\frac{q}{q-1}} \leq \|v\nabla \log u\|_2 \|v\|_p<\infty$. 

Lemma \ref{lem:ibp} now allows us to integrate by parts to get \begin{equation}\int_\Sigma \langle \nabla v^2, \nabla \log u\rangle \rho = -\int_\Sigma v^2\mathcal{L}\log u\, \rho  = \int_\Sigma v^2(\lambda_1+|A|^2+\frac{1}{2}+|\nabla \log u|^2)\rho.\end{equation} and \begin{equation}\int_\Sigma |\nabla v|^2 \rho = -\int_\Sigma v\mathcal{L}v \,\rho  = \int_\Sigma v^2(\lambda_1+|A|^2+\frac{1}{2})\rho.\end{equation} Rearranging we find that \begin{equation}\int_\Sigma |v\nabla \log u - \nabla v|^2 \rho =0,\end{equation} hence $v\nabla \log u -\nabla v=0$ and $\frac{v}{u}$ is constant on $\Sigma$. 
\end{proof}

 \begin{lemma}
 \label{lem:lipcut}
Let $V$ be an orientable $F$-stationary $n$-varifold in $\mathbf{R}^{n+1}$ with finite entropy and $\mathcal{H}^{n-q}(\sing V)=0$ for some $q\geq 2$. Then on $\Sigma=\reg V$ we get the same $\lambda_1(\Sigma)$ by taking the infimum over Lipschitz functions $f$ on $\Sigma$ that are weighted $W^{1,2}$ and $L^p$, $p= \frac{2q}{q-2}$.
\end{lemma}
\begin{proof}
Obviously we may assume that $\lambda_1=\lambda_1(\Sigma)>-\infty$. By using the global eigenfunction produced by Lemma \ref{lem:specbot} in Lemma \ref{lem:eigestimate}, we have that $|A|f$ is weighted $L^2$. Let $\phi_R$ be as in Section \ref{sec:intsing}. We will use the test functions $f_R = f \phi_R$ in the definition of $\lambda_1$.

Now since $f$ and $|A|f$ are weighted $L^2$, dominated convergence gives that $\int_\Sigma f_R^2\rho \rightarrow \int_\Sigma f^2\rho$ and $\int_\Sigma |A|^2 f_R^2\rho \rightarrow \int_\Sigma |A|^2 f^2\rho$ as $R\rightarrow \infty$. For the gradient term we have
\begin{equation}
\label{eq:lipcut}
\int_\Sigma |\nabla f_R|^2\rho =  \int_\Sigma (\phi_R^2|\nabla f|^2 + 2\langle \nabla f,\nabla\phi_R\rangle + f^2 |\nabla \phi_R|^2  )\rho.
\end{equation}
The second and third terms on the right tend to zero as $R\rightarrow \infty$, by parts (1) and (2) of Corollary \ref{cor:fgrad} respectively. Moreover, the first term tends to $\int_\Sigma |\nabla f|^2\rho$ by dominated convergence. Thus we have shown that $\int_\Sigma |\nabla f_R|^2 \rho \rightarrow \int_\Sigma |\nabla f|^2\rho$, and the lemma follows. 
\end{proof}

\begin{proposition}
\label{prop:eigcandidate}
Let $V$ be an orientable $F$-stationary $n$-varifold in $\mathbf{R}^{n+1}$ with finite entropy and $\mathcal{H}^{n-q}(\sing V)=0$ for some $q\geq 2$. Suppose that $v\neq 0$ is a $C^2$ function on $\Sigma=\reg V$ satisfying $Lv = -\lambda  v$, which is weighted $W^{1,2}$ and weighted $L^p$, $p= \frac{2q}{q-2}$. Then $\lambda_1(\Sigma) \leq \lambda$. 
\end{proposition}
\begin{proof}
Obviously we may assume $\lambda_1=\lambda_1(\Sigma) >-\infty$.

Using the positive eigenfunction produced by Lemma \ref{lem:specbot} as in the proof of that lemma, we have by Lemma \ref{lem:eigestimate} that $|A| v$ is weighted $L^2$, and hence that $v\mathcal{L}v$ is weighted $L^1$. Again since $\frac{1}{2}+\frac{1}{p} = \frac{q-1}{q}$ we have that $\|v\nabla v\|_{\frac{q}{q-1}} \leq \|\nabla v\|_2 \|v\|_p<\infty$, so by Lemma \ref{lem:lipcut} we may use $v$ as a test function in the definition of $\lambda_1$, and moreover Lemma \ref{lem:ibp} allows us to integrate by parts: 
\begin{equation}
\int_\Sigma |\nabla v|^2 \rho =  \int_\Sigma v^2 \left(\frac{1}{2}+\lambda+ |A|^2\right)\rho.\end{equation} This implies that $\lambda_1 \leq \lambda$ as claimed.
\end{proof}

\begin{corollary}
\label{cor:ncandidate}
Let $V$ be an orientable $F$-stationary $n$-varifold in $\mathbf{R}^{n+1}$ with finite entropy and $\mathcal{H}^{n-q}(\sing V)=0$ for some $q>2$. Then $\lambda_1(V) \leq -\frac{1}{2}$, with equality if and only if $\supp V$ is a hyperplane. 
\end{corollary}
\begin{proof}
Clearly we may assume $\lambda_1>-\infty$, and by Lemma \ref{lem:connectedness} we may assume that $\Sigma = \reg V$ is connected.  Fix a point $p\in \Sigma$ and set $v(x)=\langle \nu(p),\nu(x)\rangle$. Then $|v|\leq 1$ is bounded, and using the positive eigenfunction from Lemma \ref{lem:specbot} for Corollary \ref{lem:AL2} we see that $|\nabla v|\leq |A|$ is weighted $L^2$. The upper bound for $\lambda_1$ then follows from Proposition \ref{prop:eigcandidate} since $Lv=\frac{1}{2}v$. Moreover, if equality holds then since $L\langle y,\nu\rangle = \frac{1}{2}\langle y,\nu\rangle$ for any fixed $y$, the uniqueness in Lemma \ref{lem:specbot} implies that $\nu$ is constant on $\Sigma$. The constancy theorem then implies that $\supp V$ is a hyperplane. 
 \end{proof}

\begin{corollary}
\label{cor:Hcandidate}
Let $V$ be an orientable $F$-stationary $n$-varifold in $\mathbf{R}^{n+1}$ with finite entropy and $\mathcal{H}^{n-q}(\sing V)=0$ for some $q>2$. If $H$ is not identically zero on $\Sigma=\reg V$, then we have $\lambda_1(\Sigma) \leq -1$, with equality if and only if $H$ does not change sign on $\Sigma$.
\end{corollary}
\begin{proof}
From the self-shrinker equation $H=\frac{1}{2}\langle x,\nu\rangle$ we see that $|H|\leq |x|$ is weighted $L^p$ for any $p\in(0,\infty)$. Moreover, differentiating the self-shrinker equation leads to $|\nabla H| \leq |A||x|$. 

Now clearly we may assume $\lambda_1>-\infty$, and by Lemma \ref{lem:connectedness} we may assume that $\Sigma = \reg V$ is connected. Then using the positive eigenfunction $u$ of Lemma \ref{lem:specbot} for Corollary \ref{lem:AL2} implies that $|\nabla H|\leq |A||x|$ is weighted $L^2$. The result follows from Proposition \ref{prop:eigcandidate} since $LH=H$. If equality holds, the uniqueness of Lemma \ref{lem:specbot} implies that $H=c u$ does not change sign.
\end{proof}

%\begin{remark}
%Lemmas \ref{lem:AL2}, Proposition \ref{prop:eigcandidate} and Corollaries \ref{cor:ncandidate} and \ref{cor:Hcandidate} can also be seen to hold with $q=2$, if one admits functions of polynomial growth in addition to $L^\infty$ functions and chooses $\epsilon=\epsilon(R)$ accordingly in Corollary \ref{cor:fgrad}.
%\end{remark}

\subsection{Constructing unstable variations}

Here we construct $F$-unstable variations when the first stability eigenvalue $\lambda_1$ is small. 
We first consider the easy case when $\lambda_1 < -\frac{3}{2}$ which does not require any assumptions on the singular set. The proof is essentially the same as in \cite[Lemma 12.4]{CMgeneric}, but we include it here for completeness.

\begin{proposition}
\label{prop:unstable32}
Let $V$ be an orientable $F$-stationary $n$-varifold in $\mathbf{R}^{n+1}$ with finite entropy and regular part $\Sigma =\reg V$. If $\lambda_1(\Sigma) < -\frac{3}{2}$, then there exists a domain $\Omega \ssubset \Sigma$ such that if $u$ is a Dirichlet eigenfunction for $\lambda_1(\Omega)$, then for any $a\in\mathbf{R}$ and any $y\in\mathbf{R}^{n+1}$ we have \begin{equation}\int_\Omega \left( -uLu+2uaH-a^2H^2+u\langle y,\nu\rangle - \frac{\langle y,\nu\rangle^2}{2}\right)\rho <0.\end{equation} Consequently, $V$ is $F$-unstable.
\end{proposition}
\begin{proof}
Since $\lambda_1(\Sigma)<-\frac{3}{2}$ we may choose a domain $\Omega \ssubset \Sigma$ so that $\lambda_1(\Omega)<-\frac{3}{2}$. Then completing the square, the left hand side above is given by
\begin{equation} \int_\Omega \left( \left(\frac{3}{2}+\lambda_1(\Omega)\right)u^2  - (u-aH)^2 - \frac{1}{2} (u-\langle y,\nu\rangle)^2\right)\rho <0.\end{equation} so we are done by the second variation formula Proposition \ref{lem:2ndvarcrit}.
\end{proof}

We now construct $F$-unstable variations when $\lambda_1 < -1$. The key, as in \cite[Section 9.2]{CMgeneric}, is to quantify an ``almost orthogonality'' between the first eigenfunction and the eigenfunction $H$, but our analysis of the cross term differs significantly - instead of estimating boundary terms arising from integration by parts, we use our chosen cutoff functions adapted to sufficiently large domains to estimate the cross term directly. To do so, we require that the singular set is small enough that we may use the previous results of this section.

\begin{proposition}
\label{prop:unstable}
Let $V$ be an orientable $F$-stationary $n$-varifold in $\mathbf{R}^{n+1}$ with finite entropy and regular part $\Sigma=\reg V$. Suppose that $\mathcal{H}^{n-4}(\sing V)=0$. If $\lambda_1(\Sigma) <-1$, then there exists a domain $\Omega \ssubset \Sigma$ such that if $u$ is a Dirichlet eigenfunction for $\lambda_1(\Omega)$, then for any $a\in\mathbf{R}$ and any $y\in\mathbf{R}^{n+1}$ we have \begin{equation}\int_\Omega \left( -uLu+2uaH-a^2H^2 +u\langle y,\nu\rangle - \frac{\langle y,\nu\rangle^2}{2}\right)\rho <0.\end{equation} Consequently, $V$ is $F$-unstable.
\end{proposition}
\begin{proof}
As before we can absorb the cross term $u\langle y,\nu\rangle$ using $-\frac{1}{2}u^2$ and $-\frac{1}{2}\langle y,\nu\rangle^2$, so the left hand side is bounded above by \begin{equation}\label{eq:unstablegoal}\int_\Omega \left( \left(\frac{1}{2}+\lambda_1(\Omega)\right)u^2 + 2uaH-a^2H^2\right)\rho.\end{equation} If $H$ is identically zero on $\Sigma$ then we are done, so henceforth we assume this is not the case.

By Proposition \ref{prop:unstable32} we may assume $-\frac{3}{2}\leq \lambda_1(\Sigma)<-1$. Also by Lemma \ref{lem:connectedness} we may assume that $\Sigma$ is connected. We now claim that we can find a domain $\Omega\ssubset \Sigma$ with $\lambda_1(\Omega)<-1$ and for which the cross term can be absorbed by: \begin{equation}\label{eq:uHgoal}\left(\int_\Omega uH\rho\right)^2 \leq \frac{1}{2}\left(\int_\Omega H^2\rho\right) \left(\int_\Omega u^2\rho\right).\end{equation} 

Given the claim, the proof proceeds by again completing the square: Using (\ref{eq:uHgoal}) to bound the cross term, the expression (\ref{eq:unstablegoal}) is bounded above by \begin{equation} \left(1+\lambda_1(\Omega)\right)\left(\int_\Omega u^2\rho\right) - \left(\frac{1}{\sqrt{2}}\left(\int_\Omega u^2\rho\right)^\frac{1}{2} - |a|\left(\int_\Omega H^2 \rho\right)^\frac{1}{2}\right)^2<0,\end{equation} which is strictly negative since $\lambda_1(\Omega)<-1$. This implies that $V$ is $F$-unstable by the second variation formula Proposition \ref{lem:2ndvarcrit}. 

To prove the claim, we let $R>4$, set $\epsilon= R^{-3}$ and cover the singular set as in Section \ref{sec:cutoff}: $\sing V \cap B_{R} \subset \bigcup_{i=1}^m B_{r_i}(p_i)$, with $\sum_i r_i^{n-4} <\epsilon $ and $r_i<1$ for each $i$. Now we let $\phi_R=\phi_{R,\epsilon}$ be as in Section \ref{sec:cutoff} and take a domain $\Omega=\Omega_R $ such that \begin{equation}\label{eq:gooddomain} \supp \phi_R \subset \Omega_R\ssubset  \Sigma\cap B_R .\end{equation} 

Then the $\Omega=\Omega_R$ must exhaust $\Sigma$ as $R\rightarrow \infty$, so by domain monotonicity of the first eigenvalue there exists a $\delta_0>0$ such that \begin{equation}\label{eq:uHclaim1}\lambda_1(\Omega)\leq -1-\delta_0\end{equation} for any $R$ sufficiently large.

To get (\ref{eq:uHgoal}) we give ourselves some room using the cutoff function $\phi_R^2\leq 1$, \begin{eqnarray}\label{eq:uHcross} \left|\int_\Omega uH\rho\right| &=& \left|\int_\Omega (uH\phi_R^2\rho + uH(1-\phi_R^2)\rho)\right| \\\nonumber&\leq& \left|\int_\Omega uH\phi_R^2 \rho\right| + \int_{\Omega \cap \supp(1-\phi_R^2)} |uH|\rho\\\nonumber&\leq& \left|\int_\Omega uH\phi_R^2 \rho\right| + \left(\int_\Omega u^2\rho\right)^\frac{1}{2}\left(\int_{\Omega \cap \supp(1-\phi_R^2)} H^2\rho\right)^\frac{1}{2}.
\end{eqnarray}

We can crudely estimate using $|H|\leq |x|\leq R$ on $B_R$ that \begin{eqnarray} \int_{\Omega \cap \supp(1-\phi_R^2)} H^2\rho &\leq& R^2\left(\int_{\Sigma\cap B_R \setminus B_{R-3}} \rho + \sum_{i=1}^m \int_{\Sigma \cap B_{3r_i}(p_i) } \rho\right)\\\nonumber &\leq& C_V R^2\left( R^{n}\e^{-\frac{(R-3)^2}{4}} + \sum_{i=1}^m 3^n C_V r_i^n\right) \\\nonumber&\leq&C_V (R^{n+2}\e^{-\frac{(R-3)^2}{4}} + 3^n R^2 \epsilon)  ,\end{eqnarray} 
where $C_V$ is the volume growth constant, and we have used that the $r_i <1$. 

For the other term, we note that \begin{equation}H\mathcal{L}u - u\mathcal{L}H = HLu-uLH = (-\lambda_1(\Omega)-1)uH\end{equation} on $\Omega$. Setting $\alpha = -\lambda_1(\Omega)-1 \in [\delta_0,\frac{1}{2}]$, we then have \begin{eqnarray}
\int_\Omega uH\phi_R^2 \rho &=& \frac{1}{\alpha}\int_\Omega \phi_R^2(H\mathcal{L}u-u\mathcal{L}H)\rho \\\nonumber&=& \frac{2}{\alpha}\int_\Omega \phi_R\langle \nabla \phi_R, u\nabla H-H\nabla u\rangle \rho,\end{eqnarray} where we integrated by parts for the second equality. Therefore  \begin{equation} \left| \int_\Omega uH\phi_R^2 \rho \right| \leq  \frac{2}{\alpha} \int_\Omega \phi_R|\nabla \phi_R|(|u\nabla H|+|H\nabla u|)\rho.
\end{equation}

We estimate the gradient terms as follows: First, Cauchy-Schwarz gives \begin{equation}\int_\Omega \phi_R |\nabla \phi_R| |u\nabla H| \rho\leq \left(\int_\Omega u^2 \rho\right)^\frac{1}{2}\left( \int_\Omega \phi_R^2|\nabla \phi_R|^2|\nabla H|^2\rho \right)^\frac{1}{2}.\end{equation} Using $|\nabla H|\leq |A||x|\leq |A| R$ on $B_R$, we have \begin{equation} \int_\Omega \phi_R^2|\nabla \phi_R|^2|\nabla H|^2\rho \leq R^2\int_\Omega \phi_R^2|\nabla \phi_R|^2|A|^2\rho.\end{equation}

For the second gradient term, since $u$ is a first eigenfunction of $L$ on $\Omega$, we may assume without loss of generality that $u>0$ on $\Omega$. Cauchy-Schwarz then gives \begin{equation}\int_\Omega \phi_R |\nabla \phi_R| |H\nabla u| \rho\leq \left(\int_\Omega u^2 \rho\right)^\frac{1}{2}\left( \int_\Omega H^2 \phi_R^2|\nabla \phi_R|^2\frac{|\nabla u|^2}{u^2}\rho \right)^\frac{1}{2}.\end{equation} Again using $|H|\leq |x|\leq R$ on $B_R$, we have \begin{equation} \int_\Omega  H^2 \phi_R^2 |\nabla \phi_R|^2\frac{|\nabla u|^2}{u^2}\rho \leq R^2\int_\Omega \phi_R^2|\nabla \phi_R|^2\frac{|\nabla u|^2}{u^2}\rho.\end{equation}

But now by Lemma \ref{lem:AL2grad}, since $|\lambda_1(\Omega)|\leq |\lambda_1(\Sigma)| \leq \frac{3}{2}$, we have \begin{equation}
\int_\Omega (|A|^2+|\nabla \log u|^2)\phi_R^2 |\nabla \phi_R|^2 \rho \leq 268 C_V (R^n \e^{-\frac{(R-4)^2}{4}} + 4^{n}\epsilon).
\end{equation}

Putting all our estimates into (\ref{eq:uHcross}), using that $\alpha\geq \delta_0$, we obtain that \begin{equation} \frac{\left|\int_\Omega uH\rho \right|}{\left(\int_\Omega u^2\rho\right)^\frac{1}{2}} \leq C\left(R^{n+2} \e^{-\frac{(R-4)^2}{4}} + 4^{n}R^2\epsilon\right)^\frac{1}{2},\end{equation} where $C= \left(1+ \frac{2\sqrt{268}}{\delta_0}\right)\sqrt{C_V}$ does not depend on $R$. Since we chose $\epsilon = R^{-3}$, the right hand side tends to zero as $R\rightarrow\infty$. This shows that we can make $\frac{\left|\int_\Omega uH\rho \right|}{\left(\int_\Omega u^2\rho\right)^\frac{1}{2}}$ as small as we like by choosing $R$ large. But since $H$ is not identically zero, and since the $\Omega_R$ form an exhaustion of $ \Sigma$, we see that $\int_\Omega H^2 \rho$ has a uniform positive lower bound $\delta_1^2$ for sufficiently large $R$. Choosing $R$ large enough so that $\frac{\left|\int_\Omega uH\rho \right|}{\left(\int_\Omega u^2\rho\right)^\frac{1}{2}} < \frac{1}{\sqrt{2}}\delta_1$ will satisfy the condition (\ref{eq:uHgoal}). Together with (\ref{eq:uHclaim1}) this establishes the claim and thus concludes the proof. 

\end{proof}

Finally, we briefly record the construction of $F$-unstable variations of stationary cones. 

\begin{proposition}
\label{prop:unstablecone}
Let $V$ be an orientable stationary $n$-cone in $\mathbf{R}^{n+1}$ so that $H=0$ on $\Sigma=\reg V$. If $\lambda_1(\Sigma) <-\frac{1}{2}$, then there exists a domain $\Omega \ssubset \Sigma$ such that if $u$ is a Dirichlet eigenfunction for $\lambda_1(\Omega)$, then for any $y\in\mathbf{R}^{n+1}$ we have \begin{equation}\int_\Omega \left( -uLu+u\langle y,\nu\rangle - \frac{\langle y,\nu\rangle^2}{2}\right)\rho <0.\end{equation} Consequently, $V$ is $F$-unstable.
\end{proposition}
\begin{proof}
Since $\lambda_1(\Sigma)<-\frac{1}{2}$ we may choose a domain $\Omega \ssubset \Sigma$ so that $\lambda_1(\Omega)<-\frac{1}{2}$. Completing the square, the left hand side is bounded above by
$\left(\frac{1}{2}+\lambda_1(\Omega)\right) \int_\Omega  u^2 \rho <0,$ which implies that $\Sigma$ is $F$-unstable by the second variation formula Proposition \ref{lem:2ndvarcrit}, since $H=0$ on $\Sigma$. \end{proof}

\section{Mean convex singular self-shrinkers}
\label{sec:meanconvex}

Throughout this section $\Sigma$ denotes the regular part of an orientable $F$-stationary $n$-varifold $V$ in $\mathbf{R}^{n+1}$ with Euclidean volume growth. The goal is to extend the classification of mean convex self-shrinkers due to Huisken \cite{huisken90} and Colding-Minicozzi \cite{CMgeneric} to the singular setting. 

By Lemma \ref{lem:specbotlower}, if $H>0$ on $\Sigma$, then $\lambda_1(\Sigma) \geq -1$, so again some of the hypotheses on the singular set in this section may be weakened using the regularity theory Proposition \ref{prop:regularity}. We continue to state the results with the stronger hypotheses to clarify the dependence on the size of the singular set. 
We will need the following Simons-type inequality for self-shrinkers:

\begin{lemma}[\cite{CMgeneric}, Lemma 10.8]
On any smooth orientable self-shrinker we have $LA=A$. Hence, if $|A|$ does not vanish at a point then at that point one has \begin{equation}\label{eq:simons}L|A| = |A| + \frac{|\nabla A|^2-|\nabla|A||^2}{|A|}\geq |A|.\end{equation}
\end{lemma}

We now adapt the Schoen-Simon-Yau \cite{SSY} argument to improve our control on $|A|$. 

\begin{lemma}
Suppose that $\mathcal{H}^{n-4}(\sing V)=0$. If $H>0$ on $\Sigma=\reg V$ then $|A|$ is weighted $L^4$ and $|\nabla |A||$, $|\nabla A|$ are weighted $L^2$. 
\end{lemma}
\begin{proof}
First, for $\eta$ with compact support in $\Sigma$, integrating $|A|^2\eta^2 \log H$ by parts as in Lemma \ref{lem:specbotlower} and using the absorbing inequality (twice) gives
\begin{equation}
\label{eq:ssy1}
\int_\Sigma |A|^4\eta^2\rho \leq (1+a) \int_\Sigma |\nabla |A||^2\eta^2\rho +\int_\Sigma |A|^2\left((1+a^{-1})|\nabla \eta|^2 + \frac{1}{2}\eta^2\right) \rho,
\end{equation}
where $a$ is an arbitrary positive number to be chosen later. 

Second, it follows from the Simons-type inequality (\ref{eq:simons}) and Colding-Minicozzi's Kato inequality \cite[Lemma 10.2]{CMgeneric} that
\begin{equation}
\label{eq:ssy2}
\int_\Sigma |A|^4\eta^2\rho + \int_\Sigma \left(\frac{2n}{n+1}|\nabla H|^2\eta^2 +a^{-1}|A|^2|\nabla \eta|^2\right)\rho \geq \left(1+\frac{2}{n+1}-a\right)\int_\Sigma |\nabla|A||^2\eta^2\rho. 
\end{equation}

Combining (\ref{eq:ssy1}) and (\ref{eq:ssy2}) then gives
\begin{equation}
\int_\Sigma |A|^4 \eta^2\rho \leq \frac{1+a}{1+\frac{2}{n+1}-a} \int_\Sigma |A|^4\eta^2 \rho + C_{n,a} \int_\Sigma (|\nabla H|^2\eta^2+|A|^2\eta^2+|A|^2|\nabla\eta|^2)\rho.
\end{equation}

Choosing $a<\frac{1}{n+1}$ will give that the first coefficient on the right is less than 1 and thus may be absorbed on the left, therefore
\begin{equation}
\label{eq:ssy3}
\int_\Sigma |A|^4 \eta^2\rho \leq  C\int_\Sigma (|\nabla H|^2\eta^2+|A|^2\eta^2+|A|^2|\nabla\eta|^2)\rho,
\end{equation}
where $C=C(n)$. 

Let $\phi_R=\phi_{R,\epsilon}$ be as in Section \ref{sec:cutoff}. We will apply (\ref{eq:ssy3}) with $\eta = \phi_R^2$. 

As in Corollary \ref{cor:Hcandidate}, using Corollary \ref{lem:AL2} with the positive eigenfunction $H$ shows that $|A|$ and $|\nabla H|$ are weighted $L^2$. Therefore as $R\rightarrow \infty$, the first and second terms on the right will converge to the finite integrals $\int_\Sigma |\nabla H|^2\rho$ and $\int_\Sigma |A|^2\rho$ respectively. To bound the last term in (\ref{eq:ssy3}) we use Lemma \ref{lem:AL2grad} with the globally defined eigenfunction $H$, which gives
\begin{equation}
\label{eq:ssy4}
\int_\Sigma |A|^2 |\nabla \eta|^2\rho = 4\int_\Sigma |A|^2 \phi_R^2 |\nabla \phi_R|^2\rho \leq 1056 C_V ( R^n\e^{-\frac{(R-4)^2}{4}} + 4^{n}\epsilon). 
\end{equation}

Choosing $\epsilon = R^{-1}$ and taking $R\rightarrow \infty$ we see that this term tends to 0, thus we have shown that indeed $|A|$ is weighted $L^4$ by Fatou's lemma. With this fact in hand, it follows from (\ref{eq:ssy2}) that $|\nabla |A||$ is weighted $L^2$.

Finally, multiplying the identity $\mathcal{L}|A|^2=2|\nabla A|^2+|A|^2-2|A|^4$ by $\frac{1}{2}\eta^2$ and integrating by parts, we have
\begin{equation}
\int_\Sigma \eta^2(|\nabla A|^2-|A|^4)\rho \leq -\int_\Sigma 2\eta|A|\langle \nabla \eta,\nabla |A|\rangle \rho \leq \int_\Sigma (\eta^2|\nabla|A||^2 + |A|^2|\nabla \eta|^2)\rho. 
\end{equation}
Since we now know that $|\nabla |A||$ is weighted $L^2$ and that $|A|$ is weighted $L^4$, we again set $\eta=\phi_R^2$ and use (\ref{eq:ssy4}) to handle the last term; this shows that $|\nabla A|^2$ is weighted $L^2$, as desired.
\end{proof}

\begin{lemma}
\label{lem:tauconst}
Suppose that $\mathcal{H}^{n-4}(\sing V)=0$. 
If $H>0$ on $\Sigma=\reg V$, then $|A|/H$ is constant and hence $|\nabla A|^2 = |\nabla |A||^2$ on $\Sigma$. 
\end{lemma}
\begin{proof}
By Lemma \ref{lem:connectedness} we may assume that $\Sigma$ is connected.

We wish to integrate $|A|^2\mathcal{L}\log H$ and $|A|\mathcal{L}|A|$ by parts. So we check:

First, since $|A|$ is weighted $W^{1,2}$ and $L^4$ by the above lemma, using Lemma \ref{lem:eigestimate} with $H>0$ gives that $|A||\nabla \log H|$ is weighted $L^2$. Using Young's inequality we then have \begin{equation}(|A|^2|\nabla \log H|)^p = |A|^p (|A|\,|\nabla \log H|)^p \leq \frac{2-p}{2} |A|^{\frac{2p}{2-p}}+ \frac{p}{2}|A|^2|\nabla \log H|^2 .\end{equation} Since $|A|$ was weighted $L^4$ this shows that $|A|^2|\nabla\log H|$ is weighted $L^p$ for $p=\frac{4}{3} $. Since $\mathcal{L}\log H= \frac{1}{2}-|A|^2 - |\nabla \log H|^2$, we see that $|A|^2|\mathcal{L}\log H|$ is weighted $L^1$. Also \begin{equation}|\nabla |A|^2|\, |\nabla \log H| = 2|A|\,|\nabla |A|| \, |\nabla \log H| \leq |A|^2|\nabla \log H|^2 + |\nabla |A||^2\end{equation} is weighted $L^1$ since $|\nabla |A||$ was weighted $L^2$. By Lemma \ref{lem:ibp} we may now integrate $|A|^2\mathcal{L}\log H$ by parts to find that 
\begin{equation}
\label{eq:tauconst1}
\int_\Sigma \langle \nabla |A|^2,\nabla \log H\rangle \rho =  \int_\Sigma |A|^2 (|A|^2-\frac{1}{2}+|\nabla \log H|^2)\rho.
\end{equation}

Now using the Simon's equality we have that \begin{equation}|A|\mathcal{L}|A| = \frac{1}{2}|A|^2-|A|^4+|\nabla A|^2-|\nabla |A||^2\end{equation} is weighted $L^1$. We already know that $|\nabla |A||$ is weighted $L^2$, and as above we have that \begin{equation}(|A|\,|\nabla |A||)^p\leq \frac{2-p}{2} |A|^{\frac{2p}{2-p}} + \frac{p}{2}|\nabla |A||^2 .\end{equation} Again since $|A|$ is weighted $L^4$ this gives that $|A|\,|\nabla |A||$ is weighted $L^p$ for $p=\frac{4}{3}$, so we may use Lemma \ref{lem:ibp} to get that
\begin{equation}
\label{eq:tauconst2}
\int_\Sigma |\nabla |A||^2 \rho =-\int_\Sigma |A|\mathcal{L}|A|\rho \leq \int_\Sigma (|A|^4-\frac{1}{2}|A|^2)\rho.
\end{equation}

Subtracting (\ref{eq:tauconst1}) from (\ref{eq:tauconst2}) and rearranging we get \begin{equation}0\geq \int_\Sigma ||A|\nabla \log H - \nabla |A||^2\rho,\end{equation}
which implies that $|A|\nabla \log H = \nabla |A|$ and hence $|A|/H$ is constant on $ \Sigma$.

The final statement follows again from the Simons inequality (\ref{eq:simons}) since equality now must hold in the previous inequalities.
\end{proof}

We are now ready to present the proof of Theorem \ref{thm:meanconvexintro}.

\begin{theorem}
\label{thm:meanconvex}
Let $V$ be an orientable $F$-stationary $n$-varifold in $\mathbf{R}^{n+1}$ with finite entropy, and suppose that $\mathcal{H}^{n-1}(\sing V)=0$. If $H\geq 0$ on $\reg V$ then either $V$ is a stationary cone, or $\supp V$ is a generalised cylinder $\mathbf{S}^k(\sqrt{2k})\times\mathbf{R}^{n-k}$. 
\end{theorem}
\begin{proof}
Since $\mathcal{H}^{n-1}(\sing V)=0$, we may assume by Lemma \ref{lem:connectedness} that $\Sigma=\reg V$ is connected. Then since $LH=H$, by the Harnack inequality we must either have $H>0$ or $H\equiv 0$ on $\Sigma$. 
If $H\equiv 0$ on $\Sigma$ then in particular $x^\perp=0$ almost everywhere on $V$, so $V$ must be a stationary cone by Lemma \ref{lem:simon}.

Otherwise, we have $H>0$ on $\Sigma$. By Lemma \ref{lem:specbotlower}, we then have $\lambda_1(\Sigma)\geq -1$ so by the regularity theory Theorem \ref{prop:regularity}, we may assume that $\sing V$ has codimension at least 7. 

Now by Lemma \ref{lem:tauconst}, we have that $|A|/H$ is constant and $|\nabla A|^2=|\nabla |A||^2$ on $\Sigma$. The remainder of the proof of \cite[Theorem 0.17]{CMgeneric} goes through to prove that either $\nabla A\equiv 0$ on $ \Sigma$, or there are constant vectors $e_2,\cdots, e_n \in \mathbf{R}^{n+1}$ that that are tangent at every point of $\Sigma$. 

If $\nabla A\equiv 0$ on $ \Sigma$, then \cite[Theorem 4]{lawson} (which does not assume completeness) implies that $\Sigma$ is a piece of a generalised cylinder $\Sigma_0=\mathbf{S}^k(\sqrt{2k})\times\mathbf{R}^{n-k}$. Then $\supp V$ is contained in $\Sigma_0$, so by the constancy theorem we must have $\supp V = \Sigma_0$. 

On the other hand, if $e_2,\cdots, e_n \in \mathbf{R}^{n+1}$ are constant vectors tangent at every point of $\Sigma$, then by Lemma \ref{lem:simon} we have that $\mu_V = \mu_{\mathbf{R}^{n-1}}\times \mu_{\wt{V}}$, where $\wt{V}$ is an orientable $F$-stationary 1-varifold in $\mathbf{R}^2$. Since the singular set had codimension at least 7, certainly $\wt{V}$ and hence $V$ must in fact correspond to smooth complete embedded hypersurfaces. By the result of \cite[Theorem 0.17]{CMgeneric} or the remainder of its proof, we conclude that in this case $\supp V$ must be a cylinder $\mathbf{S}^1(\sqrt{2}) \times \mathbf{R}^{n-1}$. 
\end{proof}

\section{Classification of stable self-shrinkers}
\label{sec:class}

In this section we classify $F$-stable and entropy-stable singular self-shrinkers. It will be convenient to include a quick lemma verifying that there are no nontrivial stationary cones in low dimensions which satisfy the $\alpha$-structural hypothesis.

\begin{lemma}
\label{lem:lowdimcones}
Let $n\leq 2$ and suppose that $V=C(W)$ be a stationary $n$-cone in $\mathbf{R}^{n+1}$. If $V$ satisfies the $\alpha$-structural hypothesis for some $\alpha \in(0,1)$, then $\supp V$ must be a hyperplane.
\end{lemma}
\begin{proof}
If $n=1$, then the $\alpha$-structural hypothesis implies that any tangent cone to $\supp V$ consists of at most two rays, for which the only stationary configuration is a straight line. This shows that $V$ is an integer multiple of a smooth cone, hence of a line. 

If $n=2$, by dilation invariance the link must also satisfy the $\alpha$-structural hypothesis. The above argument then shows that the link $W$ is smooth. But the only smooth closed geodesics in $\mathbf{S}^2$ are the great circles, so $V$ must be a multiple of a plane. 
\end{proof}

\subsection{$F$-stable self-shrinkers}

First we classify $F$-stable self-shrinkers.

\begin{theorem}
\label{thm:Fstable}
Let $V$ be an orientable $F$-stationary $n$-varifold in $\mathbf{R}^{n+1}$ with finite entropy, that satisfies the $\alpha$-structural hypothesis for some $\alpha\in(0,\frac{1}{2})$. If $V$ is $F$-stable then $\supp V$ must be a hyperplane $\mathbf{R}^n$ or a shrinking sphere $\mathbf{S}^n(\sqrt{2n})$. 
\end{theorem}
\begin{proof}
Set $\Sigma=\reg V$. By Proposition \ref{prop:unstable32}, we may assume that $\lambda_1(V)=\lambda_1(\Sigma)\geq -\frac{3}{2}$. As such, by the regularity theory Proposition \ref{prop:regularity} and Lemma \ref{lem:connectedness}, we may assume that $\sing V$ has codimension at least 7 and hence that $\Sigma$ is connected. Since $LH=H$, the Harnack inequality gives three cases for the sign of $H$:

Case 1: $H\equiv 0 $ on $\Sigma$. If $\supp V$ is not a hyperplane $\mathbf{R}^n$, then Corollary \ref{cor:ncandidate} gives that $\lambda_1(V) < -\frac{1}{2}$. But then Proposition \ref{prop:unstablecone} shows that $V$ is $F$-unstable.

Case 2: $H$ does not vanish on $\Sigma$. In this case by Theorem \ref{thm:meanconvex} we know that $\supp V$ must be a generalised cylinder $\mathbf{S}^k(\sqrt{2k})\times \mathbf{R}^{n-k}$, $k>0$. Colding-Minicozzi showed in \cite[Theorem 0.16]{CMgeneric} that of these only the $k=n$ case is $F$-stable. 

Case 3: $H$ changes sign on $\Sigma$. In this final case, Corollary \ref{cor:Hcandidate} gives that $\lambda_1(\Sigma)<-1$. Then Proposition \ref{prop:unstable} provides an $F$-unstable variation. 
\end{proof}

We also need to classify homogenously $F$-stable stationary cones:

\begin{theorem}
\label{thm:Fstablehomog}
Let $V=C(W)$ be an orientable stationary $n$-cone in $\mathbf{R}^{n+1}$, that satisfies the $\alpha$-structural hypothesis for some $\alpha\in(0,\frac{1}{2})$. If $V$ is homogenously $F$-stable, then $\supp V$ must be a hyperplane. 
\end{theorem}
\begin{proof}
By Lemma \ref{lem:lowdimcones}, we may assume $n\geq 3$. Suppose that $W$ is not totally geodesic. We will show that $V=C(W)$ is homogenously $F$-unstable. Indeed, let $M=\reg W$ and consider a domain $\Omega \ssubset  M$. Let $u$ be a Dirichlet eigenfunction for the Jacobi operator $\wt{L}$ on $\Omega$, so that $\wt{L}u=-\kappa_1(\Omega)u$. We would like to use $u$ as our normal variation of $M$ in $\mathbf{S}^n$.

By the second variation formula for the $F$-functional on cones, Proposition \ref{lem:2ndvarcritcone}, it suffices to ensure that
\begin{equation}
\int_M \left( \kappa_1(\Omega) u^2 +2\frac{\Gamma(\frac{1+n}{2})}{\Gamma(\frac{n}{2})} u\langle y,\wt{\nu}\rangle - \frac{1}{2}\langle y,\wt{\nu}\rangle^2 \right) <0
\end{equation}
for any $y\in\mathbf{R}^{n+1}$. Completing the square we have that \begin{equation}-\frac{1}{2}\langle y,\wt{\nu}\rangle^2 + 2\frac{\Gamma(\frac{1+n}{2})}{\Gamma(\frac{n}{2})} u\langle y,\wt{\nu}\rangle \leq 2 \frac{\Gamma(\frac{1+n}{2})^2}{\Gamma(\frac{n}{2})^2} u^2 .\end{equation} 

But now $M$ is not totally geodesic and $n\geq 3$, so Theorem \ref{thm:stabilityeigsphintro} (see also \cite{zhu}) and Lemma \ref{lem:gamma} respectively give that \begin{equation}\kappa_1(M) \leq -2(n-1) < - 2 \frac{\Gamma(\frac{1+n}{2})^2}{\Gamma(\frac{n}{2})^2} .\end{equation} This  implies the existence of the desired domain $\Omega$ and thus concludes the proof.

Alternatively, having verified that the second variation formula Proposition \ref{lem:2ndvarcrit} is valid for homogenous variations (see Remark \ref{rmk:2ndvarcritcone}), we may use it directly. Setting $f(x) = |x| u(\frac{x}{|x|})$ and $\Sigma = \reg V$, as in the proof of \cite[Theorem 0.14]{CMgeneric} it suffices to ensure that 
\begin{equation}
\int_\Sigma\left( -fLf + f\langle y,\nu\rangle -\frac{1}{2}\langle y,\nu\rangle^2\right)\rho = \int_\Sigma \left( \kappa_1(\Omega) u^2 + |x| u\langle y,\nu\rangle - \frac{1}{2}\langle y,\nu\rangle^2\right)\rho <0.
\end{equation}
Estimating $2|x|u\langle y,\nu\rangle \leq \langle y,\nu\rangle^2+|x|^2u^2$, we may bound the left hand side from above by 
\begin{equation}
\label{eq:rmk1}
\int_\Sigma \left( \kappa_1(\Omega)u^2 + \frac{1}{2}|x|^2u^2\right)\rho = \int_\Sigma (\kappa_1(\Omega)+n)u^2\rho,
\end{equation}
where we have used the fact that $\int_0^\infty r^{n+1}\e^{-\frac{r^2}{4}}\d r = \frac{n}{2}\int_0^\infty r^{n-1}\e^{-\frac{r^2}{4}}\d r$. Again the fact that $\kappa_1(M) \leq -2(n-1) <-n$ completes the proof. 

\end{proof}

\begin{remark}
Similarly to Lemma \ref{lem:gamma}, using that $\lim_{n\rightarrow \infty} \frac{\Gamma(\frac{1+n}{2}) \sqrt{2}}{\Gamma(\frac{n}{2})n^{1/2}} = 1$ one may verify that $n-1 < 2 \frac{\Gamma(\frac{1+n}{2})^2}{\Gamma(\frac{n}{2})^2} < n$ for all $n$. %
The upper bound confirms that working on the link is slightly sharper than absorbing on the cone as in (\ref{eq:rmk1}). The lower bound ensures that the computation above (correctly) does not apply to the totally geodesic (planar) case. 
\end{remark}

\subsection{Entropy-stable self-shrinkers}

Finally we are ready to classify entropy-stable self-shrinkers.

\begin{theorem}
\label{thm:entropyclass}
Let $V$ be an orientable $F$-stationary $n$-varifold in $\mathbf{R}^{n+1}$ with finite entropy, that satisfies the $\alpha$-structural hypothesis for some $\alpha\in(0,\frac{1}{2})$. Assume that $V$ is not a cone. 

If $\supp V$ is not a generalised cylinder $\mathbf{S}^k(\sqrt{2k})\times\mathbf{R}^{n-k}$, then $V$ is entropy-unstable. Furthermore, if $V$ does not split off a line and if $\supp V$ is not the shrinking sphere $\mathbf{S}^n(\sqrt{2n})$, then the unstable variation can be taken to have compact support away from $\sing V$.  
\end{theorem}
\begin{proof}
First suppose that $V$ does not split off a line. If $V$ is $F$-stable then the classification of singular $F$-stable self-shrinkers Theorem \ref{thm:Fstable} gives that $\supp V$ must be a hyperplane $\mathbf{R}^n$ or the shrinking sphere $\mathbf{S}^n(\sqrt{2n})$. On the other hand, if $V$ is $F$-unstable then by Theorem \ref{thm:entropyF} it is entropy-unstable with respect to compactly supported variations. 

Now suppose that $\mu_V = \mu_{\mathbf{R}^{n-k}}\times \mu_{\wt{V}}$, where $\wt{V}$ is an orientable $F$-stationary $k$-varifold in $\mathbf{R}^{k+1}$ that does not split off a line. Then $\Lambda(V)=\Lambda(\wt{V})$. But by the above, if $\wt{V}$ is not spherical then it is entropy-unstable, and the induced (translation-invariant) variation of $V$ will also be entropy-unstable.
\end{proof}

\begin{theorem}
\label{thm:entropyclasscone}
Let $V=C(W)$ be an orientable stationary $n$-cone in $\mathbf{R}^{n+1}$, that satisfies the $\alpha$-structural hypothesis for some $\alpha\in(0,\frac{1}{2})$. If $\supp V$ is not a hyperplane $\mathbf{R}^n$, then $V$ is entropy-unstable under a homogenous variation induced by variation of the link $W$ away from its singular set.
\end{theorem}
\begin{proof}
By Lemma \ref{lem:lowdimcones}, we may assume $n\geq 3$. If $V$ is homogenously $F$-stable, then by Theorem \ref{thm:Fstablehomog}, $\supp V$ must then be a hyperplane $\mathbf{R}^n$. On the other hand, if $V$ is homogenously $F$-unstable, then by Theorem \ref{thm:entropyFcone} it is entropy-unstable under the corresponding homogenous variation. 
\end{proof}

\begin{remark}
It may be useful contextually to recall that any dilation-invariant or translation-invariant self-shrinker is entropy-stable amongst compactly supported variations, since we may shift the Gaussian centre away from the variation. Therefore, the natural variations to consider, as we have above, are those with the same symmetries as the original self-shrinker. 

One may note in particular that even the area-minimising non-flat cones are entropy-unstable when we allow the class of homogenous variations. On the one hand this makes sense since the area-minimising condition is only with respect to local perturbations, and there are certainly area-decreasing perturbations if again one allows homogenous variations. On the other hand, this suggests that the entropy functional may be limited in its ability to detect the dynamical stability of stationary cones under the mean curvature flow.
\end{remark}

Finally, Theorem \ref{thm:entropyclassintro} is simply the combination of Theorems \ref{thm:entropyclass} and \ref{thm:entropyclasscone}. We also observe:

\begin{remark}
\label{rmk:multiplicity}
In Definition \ref{def:normalvar} we considered deformations by certain ambient vector fields; in particular for higher multiplicity varifolds this did not allow the sheets to come apart. It is easy to verify that two distinct parallel planes together have entropy strictly less than 2, and similarly two distinct concentric spheres together have entropy strictly less than twice that of a single sphere. Thus, if the sheets are allowed to separate, it follows that higher multiplicity cylinders are also entropy-unstable in that sense. 
\end{remark}

\bibliographystyle{plain}
\bibliography{shrinkerstabilitymainv12}
\end{document}